\documentclass[12pt,reqno]{amsart}
\usepackage{amsmath}
\usepackage[dvips]{graphicx}
\usepackage{amsfonts}
\usepackage{amssymb}
\usepackage{latexsym}
\graphicspath{{Img/}}
\newtheorem{theorem}{Theorem}
\theoremstyle{plain}

\newtheorem{definition}{Definition}
\newtheorem{example}{Example}

\newtheorem{lemma}{Lemma}

\newtheorem{problem}{Problem}
\newtheorem{proposition}{Proposition}
\newtheorem{remark}{Remark}

\numberwithin{equation}{section}

\begin{document}

\title[The weighted Fermat-Frechet multitree in the $K$-Space]{Isometric embedding of a weighted Fermat-Frechet multitree for isoperimetric deformations of the boundary of a simplex to a Frechet multisimplex in the $K$-Space}
\author{Anastasios N. Zachos}

\address{University of Patras, Department of Mathematics, GR-26500 Rion, Greece}
\email{azachos@gmail.com}
\begin{abstract}
In this paper, we study the weighted Fermat-Frechet problem for a $\frac{N (N+1)}{2}-$tuple of positive real numbers determining $N$-simplexes in the $N$ dimensional $K$-Space ($N$-dimensional Euclidean space $\mathbb{R}^{N}$ if $K=0,$ the $N$-dimensional open hemisphere of radius $\frac{1}{\sqrt{K}}$ ($\mathbb{S}_{\frac{1}{\sqrt{K}}}^{N}$) if $K >0$ and the Lobachevsky space $\mathbb{H}_{K}^{N}$ of constant curvature $K$ if $K<0$). The (weighted) Fermat-Frechet problem is a new generalization of the (weighted) Fermat problem for $N$-simplexes.

We control the number of solutions (weighted Fermat trees) with respect to the weighted Fermat-Frechet problem that we call a weighted Fermat-Frechet multitree, by using some conditions for the edge lengths discovered by Dekster-Wilker. In order to construct an isometric immersion of a weighted Fermat-Frechet multitree in the $K$- Space, we use the isometric immersion of Godel-Schoenberg for $N$-simplexes in the $N$-sphere and the isometric immersion of Gromov (up to an additive constant) for weighted Fermat (Steiner) trees in the $N$-hyperbolic space $\mathbb{H}_{K}^{N}$.

Finally, we create a new variational method, which differs from Schafli's, Luo's and Milnor's techniques to differentiate the length of a geodesic arc with respect to a variable geodesic arc, in the 3$K$-Space. By applying this method, we eliminate one variable geodesic arc from a system of equations, which give the weighted Fermat-Frechet solution for a sextuple of edge lengths determining (Frechet) tetrahedra.
\end{abstract}\maketitle

\section{Introduction}
In \cite{Frechet:35}, Frechet posed the following problem in distance geometry (see also in \cite{Schoenberg:33}):

Let $a_{ik}=a_{ki}, i\ne k;i,k=0,1,2,\cdots,N$ ($a_{ik}=0,$ for $i=k$) be $\frac{N(N+1)}{2}$ given positive quantities. What are the necessary and sufficient conditions that they be the lengths of an $N$-simplex $A_{0}A_{1}A_{2}\cdots A_{N}$ in $\mathbb{R}^{N}$ ?

Schoenberg restates and solves Frechet problem in a more general form (\cite[Theorem~1]{Schoenberg:33}):

A necessary and sufficient condition that the $a_{ik}$ be the lengths of the edges of an $r$-simplex $A_{0}A_{1}A_{2}\cdots A_{N}$ lying in a Euclidean space $\mathbb{R}^{r}$ ($1\le r\le N$) but not in a $\mathbb{R}^{r-1}$ ($a_{ik}=0$ if and only if $i=k$) is

\[F(x_{1},x_{2},...,x_{N})=\frac{1}{2}\sum_{i,k=1}^{N}(a_{0i}^2+a_{0k}^2-a_{ik}^2)x_{i}x_{k}\ge 0\] and
\[rank(F(x_{1},x_{2},...,x_{N}))= r. \]

An extension of the Frechet problem for the construction of simplexes in the $(r-1)$-dimensional spherical space $\mathbb{S}_{p}^{r-1}$  of radius $\rho$ is given in \cite[Theorem~2]{Schoenberg:33}.

Necessary and sufficient conditions that the $a_{ik}$ be the $\frac{N(N-1)}{2}$ lengths (mutual spherical distances) of the edges of an $r-1$-simplex $A_{1}A_{2}\cdots A_{N}$ lying in $\mathbb{S}_{\rho}^{r-1}$ ($1< r\le N$) but not in a $\mathbb{S}_{\rho}^{r-2}$ for $i,k=1,2,\cdots, N$ are:

\[\varphi(x_{1},x_{2},...,x_{N})=\frac{1}{2}\sum_{i,k=1}^{n}\cos(\frac{a_{ik}}{\rho})x_{i}x_{k}\ge 0,\]

\[a_{ik}\le \pi \rho,\]

\[rank(\varphi(x_{1},x_{2},\cdots,x_{N}))= r. \]

Menger independently solved the Euclidean Frechet problem obtaining equations and inequalities by using determinants (\cite[Third Fundamental theorem,pp.~738-743]{Menger:31}) and Blumenthal-Garrett, Klanfer extends Menger's method of determinants for spherical simplexes in $S_{\rho}^{n}$ of radius $\rho$ (\cite{BlumenthalGa:33}, \cite{Klanfer:33}).


A difficult problem is to determine $\frac{\frac{1}{2}N(N+1)!}{(N+1)!}$
pairwise incongruent $N$-simplexes with the same $\frac{1}{2}N (N+1)$ edges $N>1$ in $\mathbb{R}^{N},$

In \cite{Blumenthal:59}, \cite{Fritz:59}, Blumenthal and Hertog, respectively, studied the Euclidean case for $N=3$ and  obtained conditions for six edges forming thirty incongruent tetrahedra in $\mathbb{R}^{3}.$

Dekster and Wilker discovered necessary and sufficient conditions for incongruent simplexes on spaces with constant curvature (\cite{DeksterWilker:87}, \cite{DeksterWilker:91}), by applying ideas and techniques taken from the theory of distance matrices first developed by Seidel (\cite{Seidel:69}) and Neumaier (\cite{Klanfer:33}). These conditions offer a generalization of the triangle inequality for simplexes.

The $N$-dimensional $K$-space ($\mathbb{E}_{K}^{N}$). is the $N$-dimensional Euclidean space $\mathbb{R}^{N}$ if $K=0,$ the $N$-dimensional open hemisphere of radius $\frac{1}{\sqrt{K}}$ ($\mathbb{S}_{\frac{1}{\sqrt{K}}}^{N}$) if $K >0$ and the Lobachevsky space $\mathbb{H}_{K}^{N}$ of constant curvature $K$ if $K<0.$

The weighted Fermat problem in the $\mathbb{E}_{K}^{N}$  refers to finding the unique point minimizing the sum of geodesic distances from this point to each point from a finite set of fixed non-collinear points multiplied by a positive real number (weight), which corresponds to each distance. The solution is referred as the weighted Fermat point. If the weighted Fermat point does not belong to this finite set of fixed points (vertices), the solution is called a weighted (floating) Fermat-Torricelli point, otherwise it is called a weighted (absorbing) Fermat-Torricelli point. The branching solution, which consists of the weighted branches that connect the weighted Fermat point with each fixed vertex is called a weighted Fermat tree.

We focus on the connection of the weighted Fermat problem with the Frechet problem  (weighted Fermat-Frechet problem) in the $\mathbb{E}_{K}^{N}.$

\begin{problem}[Weighted Fermat-Frechet problem]
Find the discrete set of weighted Fermat trees, which correspond to a given $N-$tuple of positive real numbers determining $\frac{1}{2}N (N+1)$ edges of incongruent $N$-simplexes in $\mathbb{E}_{K}^{N}.$
\end{problem}

In \cite{Blumenthal:59}, Blumenthal claimed that the maximum number of incongruent $(N+1)-$simplexes in $\mathbb{R}^{N}$, which may be obtained by a given $N-$tuple of positive real numbers determining $\frac{1}{2}N (N+1)$ edges  is  $\frac{\frac{1}{2}N(N+1)!}{(N+1)!}.$

In \cite{Zachos:13} and \cite{Zachos:14}, we studied the weighted Fermat-Frechet problem for $N=3$ (geodesic triangles) in the two dimensional $K$-plane ($\mathbb{E}_{K}^{2}$)  and we found the position of the weighted Fermat-Torricelli trees, which yields the weighted Fermat-Frechet (multitree) solution.

In \cite{Zachos:16}, we found the position of the weighted Fermat-Torricelli trees for the weighted Fermat-Frechet problem for a given sextuple of positive real numbers determining the edge lengths of tetrahedra in $\mathbb{R}^{3},$ by substituting Caley-Menger determinants in some weighted volume entropy equalities for tetrahedra derived in \cite{Zach/Zou:09}.

In the present paper, we introduce a class of weighted Fermat-Torricelli trees (Fermat-Frechet multitree) for a given $\frac{N (N+1)}{2}$-tuple of positive real numbers determining $N$-simplexes in $\mathbb{E}_{K}^{N}.$ We apply ideas and methods by Godel, Schoenberg and Gromov, in order to construct isometric immersions of a weighted Fermat-Frechet multitree for isoperimetric deformations of the boundary of incongruent simplexes in simplexes inside $\mathbb{E}_{K}^{N}.$

The main results of the paper are the following:

First, we study the conditions for the solution of the weighted Fermat-Frechet problem (Fermat-Frechet multitree), which corresponds to a union of weighted Fermat-Torricelli trees for all incongruent $N$-simplexes. The family of incongruent $N$-simplexes, whose edge lengths satisfy the conditions given by Dekster-Wilker (\cite{DeksterWilker:87},\cite{DeksterWilker:91a},\cite{DeksterWilker:91},\cite{DeksterWilker:92}) form a Frechet multisimplex in $\mathbb{E}_{K}^{N}$  (Theorems~\ref{theorrn1},~\ref{TheorfloatingRN},~\ref{floatmultitreehn},~\ref{floatmultitreesn}).

Next we use the ideas of Godel-Schoenberg (\cite{Schoenberg:33}), which focus on Godel's observation of thinking of the edges of an $N$-(Euclidean or spherical)simplex to made of fexible strings and on placing in the interior a small sphere, which was gradually inflated.
When the sphere reaches a definite size, it will be tightly packed within the edges (strings) of the $N$-simplex. Our idea is to consider the case of tightly packed a weighted Fermat-Torricelli tree with respect to a boundary $(N-1)$-simplex in a proper $(N-1)$ dimensional sphere.
Thus, following Godel-Shoenberg methods reversely we derive some controlled isometric immersions of a weighted Fermat-Frechet multitree to a Frechet simplex in a larger spherical space (Theorems~\ref{godelschoenbergsn},~\ref{stabletheorem2}). The conditions of Dekster-Wilker need to be satisfied, in order to control the isometric immersions of Frechet multisimplex enriched with the corresponding weighted Fermat trees, which contains all incongruent simplexes.

We apply the theory of Gromov's isometric embedding of geodesic trees in the $N$ dimensional hyperbolic space of constant curvature $K$ $\mathbb{H}_{K}^{N}$ (\cite{Gromov:87}), in order to construct an embedding (inclusion map) of a Fermat-Steiner-Frechet multitree solution for a given $\frac{N(N+1}{2}$-tuple of edge lengths determining incongruent boundary $N$-simplexes to an associated family of Gromov isometries up to an additive constant for $N$-simplexes and ideal $N$-simplexes in $\mathbb{H}_{K}^{N}.$ The Fermat-Steiner-Frechet multitree is a union of (intermediate) weighted Fermat-Steiner trees having upto $N-2$ nodes (Fermat-Steiner points) inside the $N$-simplex and may be considered as a generalization of weighted minimal binary trees studied by Ivanov and Tuzhilin in \cite{IvanovTuzhilin:95} (Theorems~\ref{thmembedmultitree},~\ref{Gromovdynamicembedding}). Furthermore, we apply Gromov's theory of $\delta$ thin metric trees in hyperbolic spaces, in order to measure the reduction of intelligence of intermediate Fermat-Steiner trees (Number of Fermat Steiner nodes), by applying Gromov's convergence of ideal simplexes to two subsimplexes (Theorem~\ref{reductionintelligence}).

We continue, by giving a new variational method, which focus on the generalization of cosine law in the 3-$K$-Space (Theorem~\ref{a04a01a02a03}). We note that a generalization of the cosine law in $\mathbb{R}^{3}$ has been given in \cite{Zach/Zou:09}.
This variational method differs from Schafli (\cite{Schlafli:53}), Luo's (\cite{LuoFeng:08})) and Milnor's (\cite{Milnor:94}) techniques to differentiate the length of a geodesic arc with respect to a variable geodesic arc in the 3-$K-$Space. By applying this method to a system of equations that deal with the weighted Fermat-Frechet solution in $\mathbb{E}_{K}^{N},$  for $N=4,$ we eliminate one variable geodesic arc from the equations and we determine the weighted Fermat-Frechet solution for a sextuple of edge lengths for Frechet  multitetrahedra (Theorem~\ref{thm1vartetrahedron}).

Finally, we conclude with some calculations for the determination of weighted Fermat trees for tetrahedra having one or three vertices at infinity in $\mathbb{R}^{3}$ (Theorem~\ref{theor2inf}).

The paper is organized as follows:
In Section~2, we study the conditions to solve the weighted Fermat-Frechet problem in the $K-$Space (Theorems~1-~9, Proposition~1,~2,~3)

In Section~3, we construct a controlled Godel-Schoenberg isometric immersion of the weighted Fermat Frechet multitree for $N-1$ boundary Frechet multisimplex in $\mathbb{S}_{\rho_{1}}^{N-2}$ to a class of weighted simplexes in $\mathbb{S}_{\rho_{0}}^{N-1}$ for $\rho_{0}>\rho_{1}$ (Theorems~9-14).

In Section~4, we construct a Gromov isometric immersion up to an additive constant for weighted Fermat-Steiner Frechet multitrees in $\mathbb{H}_{K}^{N}$ (Theorem~15-~17)

In Section~5,we give a new variational method to derive the weighted Fermat-Frechet multitree for Frechet multitetrahedra in the $3-K-$Space (Theorem~18-19, Proposition~4).

In the last section, we present a computational method to obtain weighted Fermat-Torricelli trees with respect to a tetrahedron having one or three (ideal) vertices at infinity (Theorem~20-21, Propositions~7-8).

\section{The weighted Fermat-Frechet problem in the $N$ dimensional $K$-space}
In this section, we introduce the weighted Fermat-Frechet problem for a given $\frac{N (N+1)}{2}$-tuple of positive real numbers determining $N$-simplexes in $\mathbb{E}_{K}^{N}$ and we study the conditions to obtain a class of weighted Fermat-Torricelli trees, which form the weighted Fermat-Frechet multitree solution.

We start by describing the Dekster-Wilker domain given in \cite{DeksterWilker:87}, \cite{DeksterWilker:91a}, \cite{DeksterWilker:91},\cite{DeksterWilker:92}, which gives the maximum number of mutually incongruent simplexes for the same $\frac{N(N+1)}{2}$-tuple of edges in $\mathbb{E}_{K}^{N}.$

We denote by $\{A_{1},A_{2},\cdots, A_{N}\}$ an $(N-1)$-simplex, by $A_{0}$ an interior point and by  $\{A_{1},A_{2},\cdots,A_{0}\cdots, A_{N}\} \subset \{A_{1},A_{2},\cdots, A_{N}\}$ an $(N-1)$-simplex, which is derived by replacing the vertex $A_{j}\in \{A_{1},A_{2},\cdots, A_{N}\}$ with $A_{0}$ in $\mathbb{E}_{K}^{N}.$

A. The weighted Fermat-Frechet problem in $\mathbb{R}^{N}$\\

Dekster and Wilker use the notations $\ell=\max_{i,j}a_{ij},$ $s=\min_{i,j}a_{ij}$ and
\[ \lambda_{N}(\ell )= \left\{
\begin{array}{ll}
      \ell \sqrt{1-\frac{2(N+1)}{N(N+2)}} & for\ even\ N\ge 2, \\
      \ell \sqrt{1-\frac{2}{(N+1)}} & for\ odd\ N\ge 3

\end{array}
\right.
\]
\begin{definition}
The Dekster-Wilker Euclidean domain $DW_{\mathbb{R}^{N}}(\ell, s)$ is a closed domain in $\mathbb{R}^{2}$ between the ray $s=\ell,$
and the graph of a function $\lambda_{N}(\ell ),$ $\ell \ge 0,$ which is less than $\ell$ for $\ell \ne 0,$
\end{definition}

\begin{lemma}{Dekster-Wilker incongruent simplexes in $\mathbb{R}^{N},$ \cite[(1.3),(1.4)]{DeksterWilker:91a}}\label{DW1}

If $s\le a_{ij}\le \ell,$ for $\ell, s\in DW_{\mathbb{R}^{N}}(\ell, s),$   then $a_{ij}$ determine the edge lengths of incongruent simplexes in $\mathbb{R}^{N}.$
\end{lemma}

The characterization of solutions (absorbing/floating Fermat-Torricelli trees) of the Fermat-problem in $\mathbb{R}^{N}$ was first given by Sturm in \cite{Sturm:84} and it was extended for the weighted case by Kupitz and Martini (\cite[Chapter~II, Theorem~18.37]{BolMa/So:99}) using calculus.
These two cases of weighted Fermat trees are used to derive the weighted Fermat-Frechet solution for a given $\frac{N(N+1)}{2}-$tuple of positive real numbers $a_{ij}$ determining the edge lengths of incongruent $N$-simplexes in $\mathbb{R}^{N}.$

\begin{theorem}  [The weighted Fermat-Frechet solution for $N$-simplexes in $\mathbb{R}^{N}$]\label{theorrn1}

The weighted Fermat-Frechet solution for a given $\frac{N(N+1)}{2}$-tuple $ a_{ij},$ such that $s\le a_{ij}\le \ell,$ for $\ell, s\in DW_{\mathbb{R}^{N}}(\ell, s)$ consists of a maximum number of $\frac{\frac{1}{2}N(N+1)!}{(N+1)!}$ weighted Fermat trees, which belong to one of the following two cases:

(I) If  for each index $k\in \{1,2,3,\cdots, N\},$
\begin{equation}\label{ineqq1}
\sum_{i=1, i<j}^{N} B_{i}B_{j}\frac{a_{ik}^2+a_{jk}^2-a_{ij}^2}{2a_{ik}a_{jk}}>\frac{B_{k}^2-\sum_{i=1,i<j}^{N}B_{i}^2}{2},
\end{equation}
we obtain the weighted floating Fermat-Torricelli tree $\{a_{01},a_{02},\cdots ,a_{0N}\}.$

(II) If there is an index $k\in {1,2,\cdots,N},$ such that:
\begin{equation}\label{ineqq2}
\sum_{i=1, i<j}^{N} B_{i}B_{j}\frac{a_{ik}^2+a_{jk}^2-a_{ij}^2}{2a_{ik}a_{jk}}\le \frac{B_{k}^2-\sum_{i=1,i<j}^{N}B_{i}^2}{2},
\end{equation}
we obtain the weighted absorbing Fermat-Torricelli tree $\{a_{k1},a_{k2},\cdots ,a_{kN}\}.$


\end{theorem}

\begin{proof}

The edge lengths $a_{ij}$ yield a simplex $A_{1}A_{2}\cdots A_{N+1}$ in $\mathbb{R}^{N}.$ Let $X$ be a point in $\mathbb{R}^{N}.$ We assume that a positive number $B_{i}$ (weight) correspond to each length of the segment $[X,A_{i}],$ for $i=1,2,\cdots, N.$

By applying the method of directional derivatives by Kupitz and Martini
\cite[Chapter~II, Theorem~18.16]{BolMa/So:99} for the weighted case and by applying the cosine law in $\triangle A_{i}A_{j}A_{k},$ we get (\ref{ineqq1}) and (\ref{ineqq2}).

The existence and uniqueness of the weighted Fermat-Torricelli point $X\equiv A_{0}$ proved by convexity and compactness arguments gives the existence and uniqueness of the weighted Fermat-Torricelli trees for each simplex, which belongs to the class of incongruent boundary simplexes. Thus, the maximum number of arrengements of $\frac{N (N+1)}{2}$ edges gives  $\frac{\frac{1}{2}N(N+1)!}{(N+1)!}$ mutually incongruent simplexes, which correspond to $\frac{\frac{1}{2}N(N+1)!}{(N+1)!}$ weighted floating and absorbing Fermat-Torricelli trees.

\end{proof}

\begin{theorem}  [The weighted Euclidean Fermat-Frechet problem for large $N$] \label{theorrn2}

The weighted Fermat-Frechet solution for a given large number $N$ and $\frac{N(N+1)}{2}-$tuple of edge lengths determining simplexes, consists of a maximum number of\\ $\frac{(\frac{N(N+1)}{2})^{\frac{N(N+1)+1}{2}}}{(N+1)^{N+\frac{3}{2}}}\exp(-(N+1)(\frac{N}{2}-1)+\frac{\theta_{1}}{6N(N+1)}-\frac{\theta_{2}}{12 (N+1)})$  weighted Fermat trees, for $\theta_{1},\theta_{2}\in (0,1).$

\end{theorem}

\begin{proof}
It is a direct consequence of Theorem~\ref{theorrn1} by using Stirling's formula (\cite[6.1.38, p.~257]{AbramovitzStegun:72}):

\begin{equation}\label{Stirling1}
x! =\sqrt{2\pi} x^{x+\frac{1}{2}\exp(-x+\frac{\theta}{12 x})},
\end{equation}

for $x>0$ and $\theta \in [0,1].$

By substituting in $\ref{Stirling1},$ $x\equiv \frac{N (N+1)}{2}$ and $x\equiv N+1$ for $\theta_{1},\theta_{2}\in (0,1),$ and dividing the two derived relations gives an upper bound of $\frac{(\frac{N(N+1)}{2})^{\frac{N(N+1)+1}{2}}}{(N+1)^{N+\frac{3}{2}}}\exp(-(N+1)(\frac{N}{2}-1)+\frac{\theta_{1}}{6N(N+1)}-\frac{\theta_{2}}{12 (N+1)})$  weighted Fermat trees that correspond to the mutually incogruent simplexes in $\mathbb{R}^{N}.$

\end{proof}

The volume of an $N$-simplex $A_{1}A_{2}\cdots A_{N}$ in $\mathbb{R}^{N}$ is given by the Caley-Menger determinant in terms of edge lengths (\cite[(5.1), p.~125]{Sommerville:58}):
\begin{multline*}
   \operatorname{Vol}(A_{1}A_{2}\cdots A_{N})^{2} = \\ \frac{1}{(-1)^{N} 2^{N-1} ((N-1)!)^{2}} \begin{vmatrix}
  0 & 1 & 1 & \cdots & 1 & 1 \\
  1 & 0 & a_{12}^{2} & \cdots & a_{1(N-1)}^{2} & a_{1N}^{2} \\
  1 & a_{21}^{2} & 0 & \cdots & a_{2(N-1)}^{2} & a_{2N}^{2} \\
  . & . & . & . & . & . \\
  1 & a_{N1}^{2} & a_{N2}^{2} & \cdots & a_{N(N-1)}^{2} & 0
\end{vmatrix}.
\end{multline*}

\begin{theorem}[The Euclidean Fermat-Torricelli Frechet solution]\label{TheorfloatingRN}
The following equations depending on the variable lengths $a_{01},a_{02},\cdots , a_{0N}$ and the segments $\{a_{ij}\}$ provide a necessary condition for the determination of the weighted floating Fermat-Torricelli trees $\{a_{01},a_{02},\cdots ,a_{0N}\},$ which belong to the weighted Fermat-Frechet solution:
\begin{equation}\label{construct312n}
\sum_{i=2}^{N}\frac{B_{i}}{a_{0i}}(a_{01}^2+a_{0i}^2-a_{1i}^2)-\sum_{i=1,{i\ne
j}}^{N}\frac{B_{i}}{a_{0i}}(a_{0j}^2+a_{0i}^2-a_{ji}^2)=-2(B_{1}a_{01}-B_{j}a_{0j}),
\end{equation}
\begin{equation}\label{volcalmengRn}
\operatorname{Vol}(A_{1},A_{2},\cdots,A_{N})=\sum \operatorname{Vol} (A_{1},A_{2},\cdots,A_{0},\cdots,A_{N}),
\end{equation}
for $j=2,\cdots, N.$

\end{theorem}

\begin{proof}

It is well known that the first variational formula of a line segment $A_{0}A_{i}$ with respect to a physical parameter $A_{0}A_{j}$ (line segment), which meet at a point $A_{0}$ in $\mathbb{R}^{N}$ is given by:

\begin{equation}\label{firstvariationlinesegment}
\frac{d a_{0i}}{d a_{0j}}=\cos\angle A_{i}A_{0}A_{j},
\end{equation}
for $i,j=1,2,,\cdots, N.$
Differentiating the objective function\\
 $f(a_{01},a_{02},\cdots ,a_{0N})=\sum_{i=1}^{N}B_{i}a_{0i}$ with respect to $a_{0i},$ for $i=1,2,\cdots N,$ we get:

\begin{equation}\label{eq10i}
\sum_{i=1}^{N}B_{i}\cos(\alpha_{10i})=0,
\end{equation}
\begin{equation}\label{eq20i}
\sum_{i=1}^{N}B_{i}\cos(\alpha_{20i})=0,
\end{equation}
$\cdots$
\begin{equation}\label{eqN0i}
\sum_{i=1}^{N}B_{i}\cos(\alpha_{N0i})=0.
\end{equation}

By subtracting (\ref{eq20i})-(\ref{eqN0i}) from (\ref{eq10i}) and by applying the cosine law in $\triangle A_{j}A_{0}A_{i}$ for $j=1,2,\cdots ,N,$ we derive (\ref{construct312n}).

The interior point $A_{0}$ of $A_{1}A_{2}\cdots A_{N}=\cup{A_{1}A_{2}\cdots A_{0}\cdots A_{n}}$ gives the property of addition of Euclidean volumes of simplexes in $\mathbb{R}^{N}.$
By substituting in $\operatorname{Vol}(A_{1},A_{2},\cdots,A_{N}),$ $\operatorname{Vol} (A_{1},A_{2},\cdots,A_{0},\cdots,A_{N})$ the corresponding Caley Menger determinant depending on the edge lengths $a_{ij},$ for $i,j=0,1,2,\cdots, N,$ we obtain (\ref{volcalmengRn}).

\end{proof}

The edge lengths $a_{0i}$ are the branches of the weighted Fermat tree $\{a_{01},a_{02},\cdots,a_{0N}\}$ consist the branching solution of the weighted Fermat problem. The field of branching solutions in geometric optimization has been developed by Ivanov and Tuzhilin in \cite{IvanovTuzhilin:01b}.

\begin{lemma}[Isoperimetric inequality of a regular simplex in $\mathbb{R}^{N}$]\cite[p.~274]{Hadwiger:57}\label{regmaxvol1}
An Euclidean $N$-simplex is of maximal volume if it is regular.
\end{lemma}

\begin{proposition}\label{boundbranchtrees}
Each branch of the weighted Fermat-Frechet solution is less than the diameter $2 \sqrt{\frac{N}{a (N+1)}}$ of the circumscribed sphere of a regular simplex with equal edge length $a\equiv \frac{\sum_{i,j}a_{ij}}{\frac{N(N+1)}{2}}.$
\end{proposition}

\begin{proof}
The radius of the circumsphere of a regular $N$ simplex  in $\mathbb{R}^{N}$ is given by
$2 \sqrt{\frac{N}{a (N+1)}.}$ By applying lemma~\ref{regmaxvol1}, we derive that:
\[a_{0i} < 2 \sqrt{\frac{N}{a (N+1)}},\] for $a\equiv \frac{\sum_{i,j}a_{ij}}{\frac{N(N+1)}{2}}.$
\end{proof}

B. The weighted Fermat-Frechet problem in the hyperbolic space $\mathbb{H}_{K}^{N},$ for $K<0.$ \\

Dekster and Wilker (\cite[(1.5),(1.6)]{DeksterWilker:91a},\cite{DeksterWilker:92}) describe a class of incongruent hyperbolic simplexes in terms of $\ell, s,$ such that: $\ell=\max_{i,j}a_{ij},$ $s=\min_{i,j}a_{ij}$ and
\[ \lambda_{N}(\ell )=\]\[ \left\{
\begin{array}{ll}
  \frac{1}{K}\cosh^{-1}( \sqrt{1+2(1-\frac{2}{N})\sinh ^{2}K\frac{\ell}{2}} \sqrt{1+2(1-\frac{2}{N+2})\sinh ^{2}K\frac{\ell}{2}}) & for\ even\ N\ge 2, \\
  \frac{2}{K} \sinh^{-1}(\sqrt{1-\frac{2}{(N+1)}}\sinh K\frac{\ell}{2}) & for\ odd\ N\ge 3
\end{array}
\right.
\]

\begin{definition}
The Dekster-Wilker hyperbolic domain $DW_{\mathbb{H}_{K}^{N}}(K\ell, K s)$ is a closed domain in $\mathbb{R}^{2}$ between the ray $s=\ell,$
and the graph of a function $\lambda_{N}(\ell ),$ $\ell \ge 0,$ which is less than $K \ell$ for $K \ell \ne 0,$
\end{definition}

\begin{lemma}{Dekster-Wilker incongruent simplexes in $\mathbb{H}^{N},$ \cite[(1.5),(1.6)]{DeksterWilker:91a}}\label{DW1h}

If $s\le a_{ij}\le \ell,$ for $K \ell, K s\in DW_{\mathbb{H}_{K}^{N}}(K\ell,K s),$   then $a_{ij}$ determine the edge lengths of incongruent simplexes in $\mathbb{H}_{K}^{N}.$
\end{lemma}

In \cite[Theorem~1, Proposition~2,p.~97]{NodaSakaiMorimoto:91}, Noda, Sakai, Morimoto derive a characterization of solutions (absorbing/floating Fermat-Torricelli trees)for the Fermat-problem in $\mathbb{H}_{K}^{N}$ and more general for simply connected smooth Riemannian manifolds with non-positive sectional curvature (Hadamard manifolds).
An extension of characterizations of solutions of Fermat trees for the weighted case leads to the weighted Fermat-Frechet solution for a given $\frac{N(N+1)}{2}$-tuple of positive real numbers $a_{ij}$ determining the edge lengths of incongruent $N$-simplexes in $\mathbb{H}_{K}^{N}.$

We cannot reformulate Theorem~\ref{TheorfloatingRN} for the hyperbolic case, by using directly the addition property of Volumes of hyperbolic simplexes. In \cite[Theorem, p.~200]{Milnor:94}, Milnor shows that the addition property of the Lobachevsky function, which is connected with hyperbolic volume holds for ideal (having their vertices at infinity) hyperbolic 3-simplexes and computes by using power series the volume of an ideal regular $N$-simplex
\cite[pp.~206-207]{Milnor:94}.

In \cite{HaagerupMunhkholm:81}, Haagerup and Munkholm proved an isoperimetric inequality of hyperbolic simplexes in $\mathbb{H}^{N}.$
\begin{lemma}{Isoperimetric inequality of a hyperbolic regular simplex in $\mathbb{H}_{K}^{N},$ \cite{HaagerupMunhkholm:81}}\label{regmaxvol1h}
A hyperbolic $N$-simplex is of maximal volume if it is ideal and regular.
\end{lemma}


\begin{theorem}  [The weighted Fermat-Frechet problem in $\mathbb{H}_{K}^{N}$]\label{thhn1}

The weighted Fermat-Frechet solution for a given $\frac{N(N+1)}{2}-$tuple $ a_{ij},$ such that $s\le a_{ij}\le \ell,$ for $\ell, s\in DW_{\mathbb{H}_{K}^{N}}(\ell, s)$ consists of a maximum number of $\frac{\frac{1}{2}N(N+1)!}{(N+1)!}$ weighted Fermat trees, which belong to one of the following two cases:

(I) If  for each index $k\in \{1,2,3,\cdots, N\},$
\begin{equation}\label{ineqq1h}
\sum_{i=1, i<j}^{N} B_{i}B_{j}\frac{\cosh Ka_{ik}\cosh Ka_{jk}-\cosh K a_{ij}}{\sinh Ka_{ik}\sinh Ka_{jk}}>\frac{B_{k}^2-\sum_{i=1,i<j}^{N}B_{i}^2}{2},
\end{equation}
we obtain the weighted floating Fermat-Torricelli tree $\{a_{01},a_{02},\cdots ,a_{0N}\}.$

(II) If there is an index $k\in {1,2,\cdots,N},$ such that:
\begin{equation}\label{ineqq2h}
\sum_{i=1, i<j}^{N} B_{i}B_{j}\frac{\cosh Ka_{ik}\cosh Ka_{jk}-\cosh Ka_{ij}}{\sinh Ka_{ik}\sinh Ka_{jk}}\le \frac{B_{k}^2-\sum_{i=1,i<j}^{N}B_{i}^2}{2},
\end{equation}
we obtain the weighted absorbing Fermat-Torricelli tree $\{a_{k1},a_{k2},\cdots ,a_{kN}\}.$

\end{theorem}

\begin{proof}
The existence and uniqueness of the weighted Fermat-Torricelli point $A_{0}$ in $\mathbb{H}_{K}^{N}$ is given by compactness and convexity arguments of the distance function $a_{0i}$ following Thurston \cite[Theorem~2.5.8]{Thurston:97} under the weighted conditions of Noda-Sakai-Morimoto gives a maximum number of weighted Fermat trees for $s\le a_{ij}\le \ell:$ $(Ks,K\ell)\in DW_{\mathbb{H}_{K}^{N}}(\ell, s).$
By substituting the hyperbolic law of cosines in $\triangle A_{i}A_{}jA_{k}$ in the weighted norm inequalities of unit tangent vectors, we obtain (\ref{ineqq1h}) and (\ref{ineqq2h}).
\end{proof}

\begin{theorem}[The hyperbolic Fermat-Torricelli Frechet solution]\label{floatmultitreehn}
The following equations depending on the variable hyperbolic edge lengths $a_{01},a_{02},\cdots , a_{0N}$ and the hyperbolic edge lengths $\{a_{ij}\}$ provide a necessary condition for the location of the weighted floating Fermat-Torricelli trees $\{a_{01},a_{02},\cdots ,a_{0N}\},$ which belong to the weighted Fermat-Frechet multitree solution:
\begin{align}\label{construct312nh}
\sum_{i=2}^{N}\frac{B_{i}}{\sinh K a_{01}\sinh K a_{0i}}(\cosh Ka_{01}\cosh Ka_{0i}-\cosh Ka_{1i})-\nonumber \\ \sum_{i=1,{i\ne
j}}^{N}\frac{B_{i}}{\sinh Ka_{0j} \sinh Ka_{0i}}(\cosh Ka_{0j}\cosh Ka_{0i}-\cosh Ka_{ji})=0,
\end{align}

\begin{equation}\label{hyperbolicvoln}
\operatorname{Vol(A_{1},A_{2},\cdots,A_{0}\cdots, A_{N})}<\sqrt{N}\sum_{k=0}^{\infty}\frac{\beta (\beta+1)\cdots (\beta+k-1)}{(N+2k)!}A_{N,k}
\end{equation}

with $\beta=\frac{N+1}{2},$  $A_{N,k}=\sum_{i_{0}+\cdots +i_{N}=k, i_{a}\ge 0}\frac{(2i_{0})!}{i_{0}!}\cdots \frac{(2i_{N})!}{i_{N}!},$ for $j=2,\cdots, N.$

\end{theorem}

\begin{proof}
By differentiating the objective function $f(a_{01},a_{02},\cdots a_{0N})=\sum_{i=1}^{N}B_{i}$ with respect to arc length and by using the first variation formula of a variable arc length $a_{0i}$ with respect to arc length $a_{01}$, and with respect to $a_{0j}$ and by subtracting the two derived relations and then substituting $\cos\angle A_{i}A_{0}A_{j}$ by the  hyperbolic law of cosines in $\triangle A_{i}A_{0}A_{j},$ we get (\ref{construct312nh}). By substituting in Lemma~\ref{regmaxvol1h} Milnor's computation of a hyperbolic regular $N$ simplex (upper bound), we derive (\ref{hyperbolicvoln}).

\end{proof}

Assume that the Frechet class of hyperbolic simplexes contains an orthosimplex $A_{N+1}^{\circ}A_{1}^{\circ}A_{2}^{\circ}\cdots A_{N}^{\circ}$ with edge lengths $a_{ij}^{\circ}:$ $s\le a_{ij}^{\circ}\le \ell:$ $(Ks,K\ell)\in DW_{\mathbb{H}_{K}^{N}}(K\ell, K s).$ An $N$ dimensional orthosimplex (Euclidean or non-Euclidean) is an $N-$simplex whose faces $F_{0},\cdots F_{N}$ satisfy: $F_{i}\perp F_{j}$ for $|i-j|\ge 2$ (\cite[p.~207]{Milnor:94}).
By using Milnor's elegant computation of volume of a hyperbolic orthosimplex, we get a reformulation of Theorem~\ref{floatmultitreehn} using the following lemma:

\begin{lemma}{Computation of the Volume of an orthosimplex in $\mathbb{H}_{K}^{N}$,\cite[p.~208]{Milnor:94}}

\begin{align}\label{volorthosimplexHn}
\operatorname{Vol}((A_{N+1}^{\circ}A_{1}^{\circ},A_{2}^{\circ},\cdots, A_{N}^{\circ})=\nonumber\\a_{1}a_{2}\cdots a_{N}\sum_{i_{1},i_{2},\cdots i_{N}=0}^{\infty}\frac{\beta (\beta+1)\cdots(\beta+k-1)a_{1}^{2i_{1}}\cdots a_{N}^{2i_{N}}}{i_{1}!i_{2}!\cdots i_{N}!(2i_{1}+\cdots+2i_{N}+N)\cdots (2i_{N}+1)}
\end{align}
where

$k=i_{1}+i_{2}+\cdots+i_{N},$ $\beta=\frac{N+1}{2}$\\

$\tanh K a_{(N+1)1}=a_{1}, \tanh K a_{12}=\frac{a_{2}}{\sqrt{1-a_{1}^{2}}},\cdots,$ \\ $\tanh K a_{(N-1)N}=\frac{a_{N}}{\sqrt{1-a_{1}^{2}-\cdots-a_{N-1}^{2}}}. $

\end{lemma}


\begin{theorem}[The hyperbolic Fermat-Torricelli Frechet solution]\label{floatmultitreehn}
The following equations depending on the variable hyperbolic edge lengths $a_{01},a_{02},\cdots , a_{0N}$ and the hyperbolic edge lengths $\{a_{ij}\}$ provide a necessary condition for the location of the weighted floating Fermat-Torricelli trees $\{a_{01},a_{02},\cdots ,a_{0N}\},$ which belong to the weighted Fermat-Frechet (orthosimplex) solution:
\begin{align*}
\sum_{i=2}^{N}\frac{B_{i}}{\sinh K a_{01}\sinh K a_{0i}}(\cosh Ka_{01}\cosh Ka_{0i}-\cosh Ka_{1i})-\nonumber \\ \sum_{i=1,{i\ne
j}}^{N}\frac{B_{i}}{\sinh Ka_{0j} \sinh Ka_{0i}}(\cosh Ka_{0j}\cosh Ka_{0i}-\cosh Ka_{ji})=0,
\end{align*}

\begin{equation}\label{hyperbolicvolnbis}
\operatorname{Vol}(A_{1},A_{2},\cdots,A_{0}\cdots, A_{N},A_{N+1})<\lambda \operatorname{Vol}(A_{N+1}^{\circ}A_{1}^{\circ},A_{2}^{\circ},\cdots A_{N}^{\circ})+m.
\end{equation}

for some $\lambda, m>0.$

\end{theorem}


Given $\{a_{01},a_{02},\cdots, a_{0N}\}$ a weighted (floating) Fermat-Torricelli tree with respect to the hyperbolic simplex $A_{1}A_{2}\cdots A_{N}$ in $\mathbb{H}_{K}^{N}$ with unknown constant curvature $K,$ which belongs to the hyperbolic weighted Fermat-Frechet multitree solution, we may compute the curvature $K$ and the radius $R$ of the sphere of the Klein model, by using the hyperbolic Caley-Menger determinant.
The hyperbolic Caley Menger determinant for $\{A_{0},A_{1},\cdots,A_{N}\}$ is given by:
\begin{multline*}
\det (A_{0},A_{1},A_{2},\cdots, A_{N})=\\
\begin{vmatrix}
  1 & \cosh Ka_{01} & \cosh Ka_{02} & \cdots  & \cosh Ka_{0N} \\
 \cosh Ka_{10}  & 1 & \cosh Ka_{12} & \cdots   & \cosh Ka_{1N} \\
  \cosh Ka_{20} & \cosh Ka_{21} & 1 & \cdots   & \cosh Ka_{2N} \\
  . & . & . & . & . & . \\
  \cosh Ka_{N0} & \cosh Ka_{N1} & \cosh Ka_{N2} & \cdots  & 1
\end{vmatrix}.
\end{multline*}

Thus, by setting $\det (A_{0},A_{1},A_{2},\cdots, A_{N})=0,$ we may derive an estimate of $K.$
In \cite{TerenceTao:19}, Terence Tao used a spherical Caley-Menger determinant to obtain estimate of the radius of the Earth.

\begin{proposition}\label{hyperpro1}
If a weighted Fermat-Frechet multitree $\{a_{01}, a_{02},\cdots, a_{0N}\}$ is given in $\mathbb{H}_{K}^{N}$ then we get the same sphere on the Klein model.
\end{proposition}

\begin{proof}
Solving $\det (A_{0},A_{1},A_{2},\cdots, A_{N})=0,$ for each weighted Fermat-Torricelli tree that consists the weighted Fermat-Torricelli multitree yields an estimate for the constant negative curvature $K$ and the corresponding radius $R$  in the Klein model.
\end{proof}

C. The weighted Fermat-Frechet problem in the spherical space $\mathbb{S}_{K}^{N},$ for $K>0$ and $R=\frac{1}{\sqrt{K}}.$


Dekster and Wilker (\cite[(3.1),(3.6),(1.3),(1.5)]{DeksterWilker:91a},\cite[pp.~9--11]{DeksterWilker:92}) introduced a class of incongruent spherical simplexes in terms of $\ell, s,$ such that: $\ell=\max_{i,j}a_{ij},$ $s=\min_{i,j}a_{ij}$ and
the following notations are used:
\[ \lambda_{N}(\ell )=\] \[ \left\{
\begin{array}{ll}
  \frac{1}{K}\arccos( \sqrt{1+2(1-\frac{2}{N})\sin^{2}K\frac{\ell}{2}} \sqrt{1+2(1-\frac{2}{N+2})\sin ^{2}K\frac{\ell}{2}}) & for\ even\ N\ge 2, \\
  \frac{2}{K} \arcsin(\sqrt{1-\frac{2}{(N+1)}}\sin K\frac{\ell}{2}) & for\ odd\ N\ge 3
\end{array}
\right.
,\]

\[m_{N}(K \ell)=\frac{2}{K}\arcsin K \sqrt {\frac{2+2(N-1)\cos K \ell}{1+(N-2) \cos K \ell}}\sin K \frac{\ell}{2}  \]

\[ \ell_{N}^{\star}=\frac{2}{K}\arcsin K \sqrt{\frac{N+1}{2N}},\]

\[a_{N}=\frac{2}{K}\arcsin \frac{K}{4} \frac{\sqrt{7N-4+\sqrt{N^2+8N}}}{N-1},\]

\[K \ OC: K \ s= K \ell, K \ell \in [0, K \ell_{N}^{\star}], \]

\[K \ CF: K \ s = K \ m_{N}(K \ell), K \ell \in [K \ell_{N}^{\star},K a_{N}], \]

\[K \ OE: K \ s = K \lambda_{N}(K\ \ell), K \ell \in [0,K \ell_{N}^{\star}], \]

\[EF: K s= K \lambda_{N}(K \ell), K\ \ell \in [K \ell_{N}^{\star}, K a_{N}]. \]

\begin{definition}{\cite{DeksterWilker:91a},\cite{DeksterWilker:92}}
The Dekster-Wilker spherical domain\\
$(K\ell, K s)\in DW_{\mathbb{S}_{K}^{N}}(K\ell, K s)$ is a closed domain in $\mathbb{R}^{2}$ bounded by the arcs $K \ OC,$ $K \ CF,$ $K \ OE,$ $K \ EF.$
\end{definition}

\begin{definition}{\cite[p.~333]{Chavel:93}}
 A set $D$ in $\mathbb{S}_{K}^{N},$ is convex if for any $X, Y \in D$ there exists
a geodesic arc $XY$ in $D$ such that $XY$ is the unique minimizer in $\mathbb{S}_{K}^{N},$ connecting $X$ to $Y.$
\end{definition}

\begin{definition}{\cite[p.~335]{Chavel:93}}
For any $X \ in \ \mathbb{S}_{K}^{N},$  the convexity radius is given by:
$Conv B(X,r) = \sup \{\rho : B(X,r) \ is \ convex \ for \ all \ r < \rho ,$
where $B(X,r)$ is a disk with center at $X$ and radius $r.$
\end{definition}

\begin{lemma}{Convexity Radius for the spherical space $\mathbb{S}_{K}^{N},$ \cite[p.~510]{Karcher:77},\cite[Appendix,Lemma~1]{Fletcher:09}}
The convexity radius $Conv B(X,r)$ for each $X \in \mathbb{S}_{K}^{N},$ is $\frac{\pi}{4\sqrt{K}}.$
\end{lemma}


\begin{theorem}  [The weighted Fermat-Frechet problem in $\mathbb{S}_{K}^{N}$]\label{Snfrechet}

The weighted Fermat-Frechet solution for a given $\frac{N(N+1)}{2}-$tuple $ a_{ij},$ such that $s\le a_{ij}\le \ell,$ for $(K\ell,K s) \in DW_{\mathbb{S}_{K}^{N}}(K \ell, K s)$ and $\ell \le \frac{\pi}{4\sqrt{K}}$ consists of a maximum number of $\frac{\frac{1}{2}N(N+1)!}{(N+1)!}$ weighted Fermat trees, which belong to one of the following two cases:

(I) If  for each index $k\in \{1,2,3,\cdots, N\},$
\begin{equation}\label{ineqq1s}
\sum_{i=1, i<j}^{N} B_{i}B_{j}\frac{-\cos Ka_{ik}\cos Ka_{jk}+\cos K a_{ij}}{\sin Ka_{ik}\sin Ka_{jk}}>\frac{B_{k}^2-\sum_{i=1,i<j}^{N}B_{i}^2}{2},
\end{equation}
we obtain the weighted floating Fermat-Torricelli tree $\{a_{01},a_{02},\cdots ,a_{0N}\}.$

(II) If there is an index $k\in {1,2,\cdots,N},$ such that:
\begin{equation}\label{ineqq2s}
\sum_{i=1, i<j}^{N} B_{i}B_{j}\frac{-\cos Ka_{ik}\cos Ka_{jk}+\cos Ka_{ij}}{\sin Ka_{ik}\sin Ka_{jk}}\le \frac{B_{k}^2-\sum_{i=1,i<j}^{N}B_{i}^2}{2},
\end{equation}
we obtain the weighted absorbing Fermat-Torricelli tree $\{a_{k1},a_{k2},\cdots ,a_{kN}\}.$

\end{theorem}

\begin{proof}
The existence and uniqueness of the weighted Fermat-Torricelli point $A_{0}$ in $\mathbb{S}_{K}^{N}$ is given by compactness and the convexity of the distance function $a_{0i}$ in \cite[(1.2.4)]{Karcher:77}, by selecting incongruent $N-$simplexes  $A_{1}A_{2}\cdots A_{N}$ in $B_{A_{i},\frac{\pi}{\sqrt{K}}}.$  Therefore, the maximum number of weighted Fermat trees for $s\le a_{ij}\le \ell:$ $(Ks,K\ell)\in DW_{\mathbb{S}_{K}^{N}}(K\ell,K s)$ is
$\frac{\frac{1}{2}N(N+1)!}{(N+1)!}.$

The gradient of the objective function $f(a_{10},a_{20},\cdots, a_{N0})=\sum_{i=1}^{N}B_{i}a_{0i}$ gives a weighted sum of outward pointing unit tangent vectors at $A_{0}$ $\nabla f(a_{10},a_{20},\cdots, a_{N0})=\sum_{i=1}^{N}B_{i}\exp_{A_{0}}^{-1}\frac{\vec{X}(A_{0},A_{i})}{a_{0i}}.$ Taking the norm of the inner product $\nabla f(a_{10},a_{20},\cdots, a_{N0}) \dot \exp_{A_{0}}^{-1}\frac{\vec{X}(A_{0},A_{i})}{a_{0i}}$ yields an extension of weighted norm conditions of Kupitz-Martini for the spherical case. By substituting the spherical law of cosines in in $\triangle A_{i}A_{}jA_{k}$ in the weighted norm inequalities of unit tangent vectors, we get (\ref{ineqq1s}) and (\ref{ineqq2s}).
\end{proof}

In \cite{Boroczsky:87}, Boroczsky proved an isoperimetric inequality of spherical simplexes in $\mathbb{S}_{K}^{N}.$
\begin{lemma}{Isoperimetric inequality of a spherical regular simplex in $\mathbb{S}_{K}^{N},$ \cite{Boroczsky:87}} \label{regmaxvol1s}
A spherical $N$ simplex is of maximal volume if it is regular.
\end{lemma}

\begin{theorem}[The spherical Fermat-Torricelli Frechet solution]\label{floatmultitreesn}
The following equations depending on the variable spherical edge lengths $a_{01},a_{02},\cdots , a_{0N}$ and the spherical edge lengths $\{a_{ij}\}$ provide a necessary condition for the location of the weighted floating Fermat-Torricelli trees $\{a_{01},a_{02},\cdots ,a_{0N}\},$ which belong to the weighted Fermat-Frechet multitree solution:
\begin{align}\label{construct312ns}
\sum_{i=2}^{N}\frac{B_{i}}{\sin K a_{01}\sin K a_{0i}}(-\cos Ka_{01}\cos Ka_{0i}+ \cos Ka_{1i})-\nonumber \\ \sum_{i=1,{i\ne
j}}^{N}\frac{B_{i}}{\sin Ka_{0j} \sin Ka_{0i}}(-\cos Ka_{0j}\cos Ka_{0i}+\cos Ka_{ji})=0,
\end{align}

\begin{equation}\label{sphericalvols}
\operatorname{Vol}(A_{1},A_{2},\cdots,A_{0}\cdots, A_{N}))<\operatorname{Vol}(C_{1},C_{2},\cdots, C_{N}))
\end{equation}

where $C_{1},C_{2},\cdots C_{N}$ is a regular spherical $N$-simplex with edge length $c_{ij}=\frac{\sum_{i,j}a_{ij}}{\frac{(N-1)N}{2}}.$

\end{theorem}

\begin{proof}
By differentiating the objective function $f(a_{01},a_{02},\cdots a_{0N})=\sum_{i=1}^{N}B_{i}$ with respect to arc length and by using the first variation formula of a variable arc length $a_{0i}$ with respect to arc length $a_{01}$, and with respect to $a_{0j}$ and by subtracting the two derived relations and then substituting $\cos\angle A_{i}A_{0}A_{j}$ by the  spherical law of cosines in $\triangle A_{i}A_{0}A_{j},$ we get (\ref{construct312ns}). Taking into account Lemma~\ref{regmaxvol1s}, we obtain :
\[\operatorname{Vol}(A_{1},A_{2},\cdots,A_{0}\cdots, A_{N}))<\operatorname{Vol}(A_{1},A_{2},\cdots, A_{N}))< \]\[\operatorname{Vol}(C_{1},C_{2},\cdots,\cdots, C_{N})),\] which gives (\ref{sphericalvols}).

\end{proof}


Given $\{a_{01},a_{02},\cdots, a_{0N}\}$ a weighted (floating) Fermat-Torricelli tree with respect to the spherical simplex $A_{1}A_{2}\cdots A_{N}$ in $\mathbb{S}_{K}^{N}$ with unknown constant curvature $K,$ which belongs to the spherical weighted Fermat-Frechet multitree, we may compute the curvature $K$ and the radius $R$ of $\mathbb{S}_{K}^{N},$ by using the spherical Caley-Menger determinant.
The spherical Caley Menger determinant for $\{A_{0},A_{1},\cdots,A_{N}\}$ is given by:
\begin{multline*}
\det (A_{0},A_{1},A_{2},\cdots, A_{N})=\\
\begin{vmatrix}
  1 & \cos Ka_{01} & \cos Ka_{02} & \cdots  & \cos Ka_{0N} \\
 \cos Ka_{10}  & 1 & \cos Ka_{12} & \cdots   & \cos Ka_{1N} \\
  \cos Ka_{20} & \cos Ka_{21} & 1 & \cdots   & \cos Ka_{2N} \\
  . & . & . & . & . & . \\
  \cos Ka_{N0} & \cos Ka_{N1} & \cos Ka_{N2} & \cdots  & 1
\end{vmatrix}.
\end{multline*}

By setting $\det (A_{0},A_{1},A_{2},\cdots, A_{N})=0,$ we derive proposition~\ref{spherepro1}, which gives an estimate of $K,$ (see in \cite{TerenceTao:19} for $N=3.$)

\begin{proposition}\label{spherepro1}
If a weighted Fermat-Frechet multitree $\{a_{01}, a_{02},\cdots, a_{0N}\}$ is given in $\mathbb{S}_{K}^{N}$ then the intersection of the solution set\\ $\det (A_{0},A_{1},A_{2},\cdots, A_{N})=0,$ for all incongruent spherical simplexes with respect to $K$ gives the same sphere with radius $\frac{1}{\sqrt{K}}.$
\end{proposition}


\section{A controlled Godel-Schoenberg's isometric immersion of weighted Fermat trees in $\mathbb{S}_{\rho_{1}}^{N-2}$ to a weighted simplex in $\mathbb{S}_{\rho_{0}}^{N-1}$}
In this section, we use Godel Schoenberg's techniques, in order to construct isometric immersions of weighted Fermat trees in $\mathbb{S}_{\rho_{1}}^{N-2}$ to a weighted simplex in $\mathbb{S}_{\rho_{0}}^{N-1}.$

In \cite{Godel:33} (see also in \cite[Footnote~5, p.~730]{Schoenberg:33}), Godel considered the edges of a tetrahedron in $\mathbb{R}^{3}$ to be made of flexible strings and placed in the interior of the tetrahedron a small sphere, which was gradually inflated. After time $t,$ this sphere will become tightly packed within the six edges of the tetrahedron. We are interested in the case where the three edges of the tetrahedron correspond to a boundary spherical triangle and the other edges are the three branches of the weighted Fermat (geodesic) tree, which meet at an interior Fermat point.

Let $\triangle A_{1}A_{2}A_{3}$ be a spherical triangle in the open hemisphere $\mathbb{S}_{K}^{2}$ of radius $R=\frac{1}{\sqrt{K}}\equiv \frac{1}{\kappa}.$ We denote by $A_{0}$ the weighted Fermat point inside $\triangle A_{1}A_{2}A_{3}$ $\alpha_{ijk}\equiv \angle A_{i}A_{j}A_{k}$ and by $\{A_{0}A_{1}, A_{0}A_{2}, A_{0}A_{3}\} $ the weighted Fermat tree, such that a positive real number (weight) $B_{r}$ corresponds to each geodesic branch $A_{0}A_{r},$ for $i,j,k=0,1,2,3, r=1,2,3.$
\begin{theorem}[Godel's isometric immersion of a weighted Fermat tree with respect to a boundary spherical triangle to a weighted tetrahedron in $\mathbb{R}^{3}$]\label{Godelr3}
If we select a triad of weights $\{B_{1},B_{2},B_{3}\},$ which satisfy
\begin{eqnarray}\label{calc1}
\lefteqn{1-(\frac{\sin(\kappa
a_{13})}{\sin(\arccos(\frac{B_{2}^2-B_{1}^2-B_{3}^2}{2B_{1}B_{3}}))}\sin(\alpha_{213}-\arccos(\frac{\cos(\kappa
a_{02})-\cos(\kappa a_{01})\cos(\kappa a_{12})}{\sin(\kappa
a_{01})\sin(\kappa a_{12})}))))^2=
{}}\nonumber\\
&&{}(\cos(\kappa a_{01})\cos(\kappa a_{13})+\sin(\kappa
a_{13})\cos(\alpha_{213})\frac{\cos(\kappa a_{02})-\cos(\kappa
a_{01})\cos(\kappa a_{12})}{\sin(\kappa a_{12})}+ {}\nonumber\\
&&{}+\sin(\kappa a_{01})\sin(\kappa a_{13})\sin(\alpha_{213})\cdot
{}\nonumber\\
&&{}\cdot \sqrt{1-(\frac{\cos(\kappa a_{02})-\cos(\kappa
a_{01})\cos(\kappa a_{12})}{\sin(\kappa a_{01})\sin(\kappa
a_{12})})^2}\quad\quad\quad)^2,
\end{eqnarray}

\begin{equation}\label{calc2}
(\cos(\kappa a_{12})-\cos(\kappa a_{01})\cos(\kappa
a_{02}))^2=(\sin(\kappa a_{01})\sin(\kappa
a_{02})\frac{B_{3}^2-B_{1}^2-B_{2}^2}{2B_{1}B_{2}})^2
\end{equation}

such that $\{a_{01},a_{02},a_{03},a_{12},a_{13},a_{23}\}\in DW_{\mathbb{R}^{2}}(\ell, s)$
then the weighted Fermat tree $\{a_{01},a_{02},a_{03}\}$ and the edges of the boundary spherical triangle are isometrically immersed to a weighted tetrahedron in $\mathbb{R}^{3}$ with corresponding edges $\{a_{01},a_{02},a_{03},a_{12},a_{13},a_{23}\}$ and corresponding weights $\{B_{1},B_{2},B_{3},1,1,1\}.$
\end{theorem}

\begin{proof}
We construct an isometric immersion of a Fermat tree with respect to $\triangle A_{1}A_{2}A_{2}$ following Godel's observation of a tetrahedron with six flexible edges tightly pack to a sphere of given radius $R=\frac{1}{\sqrt{K}}=\frac{1}{\kappa}$ and taking into account two equations derived by \cite[Theorem~2.4, (2.23),(2.24)]{Zachos:13}, which determine the location of the weighted Fermat-tree. By restricting the edge lengths $\{a_{01},a_{02},a_{03},a_{12},a_{13},a_{23}\}\in DW_{\mathbb{R}^{2}}(\ell, s),$ we obtain a weighted tetrahedron $A_{0}A_{1}A_{2}A_{3}$ in $\mathbb{R}^{3}.$
\end{proof}

\begin{remark}
We note that the two equations (\ref{calc1}) and (\ref{calc2}) , which give the location of the weighted Fermat  point $\triangle A_{1}A_{2}A_{3}$ on the K-plane (two dimensional sphere $S_{K}^{2}$ and two dimensional
hyperbolic plane $H_{K}^{2}$) can be merged to a rational algebraic
equation, which depends only on the variable $z\equiv\sin\angle
A_{0}A_{1}A_{3}$ (see in \cite[Theorem~2.4, (2.8)]{Zachos:14}).
\end{remark}

In \cite[Theorem~3,p.~728]{Schoenberg:33}, Schoenberg obtained an isometric immersion of an $(N-1)$ spherical simplex in $\mathbb{S}_{\rho_{0}}^{N-1}$ to $\mathbb{S}_{\rho_{1}}^{N-2}$ by proving the existence of a radius $\rho_{1}\le \rho_{0}.$ Thus, by applying Schoenberg's isometric immersion, we derive the following reformulation of Theorem~\ref{Godelr3} for $N=4.$

\begin{theorem}[Schoenberg's isometric immersion of a weighted Fermat tree with respect to a boundary spherical triangle in $\mathbb{S}_{\rho_{1}}^{2}$ to a weighted tetrahedron in $\mathbb{S}_{\rho_{0}}^{3}$]\label{Godels3}
If we select a triad of weights $\{B_{1},B_{2},B_{3}\},$ which satisfy
\begin{eqnarray}\label{calcrho1}
\lefteqn{1-(\frac{\sin(\kappa_{1}
a_{13})}{\sin(\arccos(\frac{B_{2}^2-B_{1}^2-B_{3}^2}{2B_{1}B_{3}}))}\sin(\alpha_{213}-\arccos(\frac{\cos(\kappa_{1}
a_{02})-\cos(\kappa_{1} a_{01})\cos(\kappa_{1} a_{12})}{\sin(\kappa_{1}
a_{01})\sin(\kappa_{1} a_{12})}))))^2=
{}}\nonumber\\
&&{}(\cos(\kappa_{1} a_{01})\cos(\kappa_{1} a_{13})+\sin(\kappa_{1}
a_{13})\cos(\alpha_{213})\frac{\cos(\kappa_{1} a_{02})-\cos(\kappa_{1}
a_{01})\cos(\kappa_{1} a_{12})}{\sin(\kappa_{1} a_{12})}+ {}\nonumber\\
&&{}+\sin(\kappa_{1} a_{01})\sin(\kappa_{1} a_{13})\sin(\alpha_{213})\cdot
{}\nonumber\\
&&{}\cdot \sqrt{1-(\frac{\cos(\kappa_{1} a_{02})-\cos(\kappa_{1}
a_{01})\cos(\kappa_{1} a_{12})}{\sin(\kappa_{1} a_{01})\sin(\kappa_{1}
a_{12})})^2}\quad\quad\quad)^2,
\end{eqnarray}

\begin{equation}\label{calcrho12}
(\cos(\kappa_{1} a_{12})-\cos(\kappa_{1} a_{01})\cos(\kappa_{1}
a_{02}))^2=(\sin(\kappa_{1} a_{01})\sin(\kappa_{1}
a_{02})\frac{B_{3}^2-B_{1}^2-B_{2}^2}{2B_{1}B_{2}})^2
\end{equation}

then there exists a $\rho_{0}> \rho_{1},$ such that the weighted Fermat tree $\{a_{01},a_{02},a_{03}\}$ and the edges of the boundary spherical triangle in $\mathbb{S}_{\rho_{1}}^{2}$ are isometrically immersed to a weighted spherical tetrahedron in $\mathbb{S}_{\rho_{0}}^{3}$ with corresponding edges $\{a_{01},a_{02},a_{03},a_{12},a_{13},a_{23}\}$ and corresponding weights $\{B_{1},B_{2},B_{3},1,1,1\},$

for $\{a_{01},a_{02},a_{03},a_{12},a_{13},a_{23}\}\in DW_{\mathbb{S}_{\rho_{0}}^{3}}(\ell, s).$

\end{theorem}

We note that we may use the weighted Fermat-Torricelli tree for an equilateral triangle having equal edge lengths $\frac{\pi}{2}$ in $\mathbb{S}_{1}^{2},$ which was given explicitly in \cite{Zachos:15}, in order to construct a weighted isometric immersion in $\mathbb{S}_{\rho_{0}}^{3}$ and in $\mathbb{R}^{3}.$ These isometric immersions need to be controlled by the conditions for the edge lengths described by the Dekster-Wilker spherical and Euclidean domain.

Let $A_{1}A_{2}\cdots A_{N-1}$ be an $(N-2)$-simplex in $\mathbb{S}_{\rho_{1}}^{N-2},$  $A_{0}$ be the weighted Fermat point of the $(N-2)$-simplex having degree $N-1$ and\\ $\{A_{1}A_{0},A_{2}A_{0},\cdots, A_{N-1}A_{0}\}$ the corresponding weighted Fermat tree. We denote by $B_{i}$ the weight, which corresponds to each geodesic branch $A_{i}A_{0}$ and by $a_{jm}$ the length of the geodesic arc $A_{j}A_{m},$ for $i=0,1,2,\cdots N-1$ and $j,m=0,1,2,\cdots, N-1.$

We denote by $\{A_{1},A_{2},\cdots,A_{0}\cdots, A_{N-1}\} \subset \{A_{1},A_{2},\cdots, A_{N-1}\}$ an $(N-2)$-simplex, which is derived by replacing the vertex\\ $A_{j}\in \{A_{1},A_{2},\cdots, A_{N-1}\}$ with $A_{0}.$

\begin{theorem}[Godel-Schoenberg's isometric immersion of weighted Fermat trees in $\mathbb{S}_{\rho_{1}}^{N-2}$ to $\mathbb{S}_{\rho_{0}}^{N-1}$]\label{godelschoenbergsn}
If we select an $(N-1)$-tuple of weights $\{B_{1},\cdots,B_{N-1}\},$ which satisfy
\begin{align}\label{construct312nsminus2}
\sum_{i=2}^{N-1}\frac{B_{i}}{\sin K a_{01}\sin K a_{0i}}(-\cos Ka_{01}\cos Ka_{0i}+ \cos Ka_{1i})-\nonumber \\ \sum_{i=1,{i\ne
j}}^{N-1}\frac{B_{i}}{\sin Ka_{0j} \sin Ka_{0i}}(-\cos Ka_{0j}\cos Ka_{0i}+\cos Ka_{ji})=0,
\end{align}

\begin{equation}\label{sphericalvolsminus2}
\operatorname{Vol}(A_{1},A_{2},\cdots,A_{0}\cdots, A_{N-1}))<\operatorname{Vol}(C_{1},C_{2},\cdots, C_{N-1}))
\end{equation}

where $C_{1},C_{2},\cdots C_{N-1}$ is a regular spherical $(N-2)$-simplex with edge length $c_{ij}=\frac{\sum_{i,j}a_{ij}}{\frac{(N-2)(N-1)}{2}}.$

then there exists a $\rho_{0}> \rho_{1},$ such that the weighted Fermat tree $\{a_{01},\cdots ,a_{0N-1}\}$ and the edges of the boundary $(N-1)$-simplex $\{a_{ij}\}$ in $\mathbb{S}_{\rho_{1}}^{N-2}$ are isometrically immersed to a weighted spherical $(N-1)$-simplex in $\mathbb{S}_{\rho_{0}}^{N-1}$ with corresponding edges $\{a_{01},\cdots,a_{0N-1},\{a_{ij}\}\}$ and corresponding weights $\{B_{1},B_{2},\cdots, B_{N-1},1,\cdots,1\},$

for the $\frac{(N-1)N}{2}$tuple of edge lengths $\{a_{01},\cdots,a_{0N-1},\{a_{ij}\}\}\in DW_{\mathbb{S}_{\rho_{0}}^{N-1}}(\ell, s).$

\end{theorem}

\begin{proof}
By applying Schoenberg isometric immersion and by taking into account the conditions of Theorem~\ref{floatmultitreesn}, we can construct an isometric immersion od weighted Fermat trees in $\mathbb{S}_{\rho_{1}}^{N-2}$ to $\mathbb{S}_{\rho_{0}}^{N-1},$ the $\frac{(N-1)N}{2}$tuple of edge lengths $\{a_{01},\cdots,a_{0N-2},\{a_{ij}\}\}\in DW_{\mathbb{S}_{\rho_{0}}^{N-1}}(\ell, s).$
\end{proof}

By substituting $\rho=\infty$ in Schoenberg's isometric immersion, we get Godel-Schoenberg's isometric immersion  (\cite[Theorem~3',p.~730]{Schoenberg:33} ), we derive Godel's isometric immersion of weighted Fermat trees for a boundary $(N-1)$-simplex in $\mathbb{S}_{\rho}^{N-1}$ to a weighted $N$-simplex $\mathbb{R}^{N}.$

We denote by $\{A_{1},A_{2},\cdots,A_{0}\cdots, A_{N}\} \subset \{A_{1},A_{2},\cdots, A_{N}\}$ an $(N-1)$-simplex, which is derived by replacing the vertex $A_{j}\in \{A_{1},A_{2},\cdots, A_{N}\}$ with $A_{0}.$

\begin{theorem}[Godel's isometric immersion of weighted Fermat trees in $\mathbb{S}_{\rho}^{N-1}$ to $\mathbb{R}^{N}$]\label{godeltheoreisomimersphere}

If we select an $N$-tuple of weights $\{B_{1},\cdots,B_{N-1},B_{N}\},$ which satisfy
\begin{align}\label{construct312nsminus2r}
\sum_{i=2}^{N}\frac{B_{i}}{\sin K a_{01}\sin K a_{0i}}(-\cos Ka_{01}\cos Ka_{0i}+ \cos Ka_{1i})-\nonumber \\ \sum_{i=1,{i\ne
j}}^{N}\frac{B_{i}}{\sin Ka_{0j} \sin Ka_{0i}}(-\cos Ka_{0j}\cos Ka_{0i}+\cos Ka_{ji})=0,
\end{align}

\begin{equation}\label{sphericalvolsminus2r}
\operatorname{Vol}(A_{1},A_{2},\cdots,A_{0}\cdots, A_{N}))<\operatorname{Vol}(C_{1},C_{2},\cdots, C_{N-1},C_{N})
\end{equation}

where $\{C_{1},C_{2},\cdots C_{N-1},C_{N}\}$ is a regular spherical $(N-1)$simplex with edge length $c_{ij}=\frac{\sum_{i,j}a_{ij}}{\frac{(N-1)N}{2}}.$

then there exists an isometric immersion of the weighted Fermat tree $\{a_{01},\cdots ,a_{0N}\}$ and the edges of the boundary $N-1$ simplex $\{a_{ij}\}$ in $\mathbb{S}_{\rho_{1}}^{N-1}$ to a weighted $N$-simplex in $\mathbb{R}^{N}$ with corresponding edges $\{a_{01},\cdots,a_{0N},\{a_{ij}\}\}$ and corresponding weights $\{B_{1},B_{2},\cdots, B_{N},1,\cdots,1\},$

for the $\frac{(N(N+1}{2}$tuple of edge lengths $\{a_{01},\cdots,a_{0N},\{a_{ij}\}\}\in DW_{\mathbb{R}^{N}}(\ell, s).$

\end{theorem}

We consider a $\frac{(N-1)N}{2}-$tuple of positive real numbers $a_{ij},$ determining the edge lengths of $(N-1)$ incongruent spherical simplexes, such that: $a_{ij}$ belong to the spherical domain of Dekster-Wilker $DW_{\mathbb{S}_{\rho_{0}}^{N-1}} (\ell,s)$ and $s\le a_{ij}\le \ell.$ Thus, all incongruent pairwise spherical $(N-1)$-simplexes may yield up to $\frac{\frac{1}{2}(N-1)N!}{N!}$ weighted Fermat-Torricelli trees for a given $N$-tuple of weights $\{B_{1},B_{2},\cdots, B_{N}.$ The union of these weighted Fermat-Torricelli trees determine the weighted Fermat-Frechet multitree (solution).

\begin{definition}[Stable isometric immersion of a variable weighted Fermat-Frechet multitree]
We call a stable isometric immersion of a variable weighted Fermat-Frechet multitree for a $\frac{(N-1)N}{2}-$tuple of positive real numbers $a_{ij},$ determining $(N-1)-$ incongruent spherical simplexes in $\mathbb{S}_{\rho}^{N-1}$ to $\mathbb{R}^{N}$ the solution set of the variable weights $\{B_{1},B_{2},\cdots,B_{N}\},$ such that each  $N-1$ spherical simplex with the corresponding variable weighted Fermat tree in $\mathbb{S}_{\rho}^{N-1}$ is isometrically immersed to an $N-$simplex in $\mathbb{R}^{N}.$
\end{definition}

\begin{definition}[Godel-Frechet multisimplex]
We call Godel-Frechet multisimplex a union of $N$-simplexes in $\mathbb{R}^{N},$ which is derived by a stable isometric immersion of a variable weighted Fermat-Frechet multitree in $\mathbb{S}_{\rho}^{N-1}.$
\end{definition}

We denote by $B_{i}(t)$ a variable weight, which correspond to the geodesic branch $A_{0}A_{i}$ of each weighted Fermat tree $\{A_{1}A_{0},c\dots, A_{N}A_{0}\},$ which forms a weighted Fermat-Frechet multitree, for $t>0,$ $i=1,2,\cdots, N.$

\begin{theorem}[Godel's isometric immersions of a weighted Fermat-Frechet multitree in $\mathbb{S}_{\rho}^{N-1}$ to $\mathbb{R}^{N}$ ]\label{stabletheorem1}

The following conditions for $B_{i}(t)$ provide a stable isometric immersion of a weighted Fermat-Frechet multisimplex in $\mathbb{S}_{\rho}^{N-1}$ to a Godel-Frechet simplex in $\mathbb{R}^{N}:$

\begin{align}\label{construct312nsminus2rvar}
\sum_{i=2}^{N}\frac{B_{i}(t)}{\sin K a_{01}\sin K a_{0i}}(-\cos Ka_{01}\cos Ka_{0i}+ \cos Ka_{1i})-\nonumber \\ \sum_{i=1,{i\ne
j}}^{N}\frac{B_{i}(t)}{\sin Ka_{0j} \sin Ka_{0i}}(-\cos Ka_{0j}\cos Ka_{0i}+\cos Ka_{ji})=0,
\end{align}

\begin{equation}\label{isopcondvariableweights}
\sum_{i=1}^{N}B_{i}(t)=1,
\end{equation}

\begin{equation}\label{sphericalvolsminus2rvar}
\operatorname{Vol}(A_{1},A_{2},\cdots,A_{0}\cdots, A_{N}))<\operatorname{Vol}(C_{1},C_{2},\cdots, C_{N}))
\end{equation}

where $C_{1},C_{2},\cdots C_{N}$ is a regular spherical $(N-1)$simplex with edge length $c_{ij}=\frac{\sum_{i,j}a_{ij}}{\frac{(N-1)N}{2}}.$


for the $\frac{(N(N+1}{2}$tuple of edge lengths $\{a_{01},\cdots,a_{0N},\{a_{ij}\}\}\in DW_{\mathbb{R}^{N}}(\ell, s).$

\end{theorem}

\begin{proof}
By substituting $B_{i}(t) \to B_{i}$ in (\ref{construct312nsminus2r}) taking into account the isoperimetric condition for the variable weights (\ref{isopcondvariableweights}) and the volume inequality (\ref{sphericalvolsminus2r}) taken from Theorem~\ref{godeltheoreisomimersphere}, we derive (\ref{construct312nsminus2rvar}), which yields the desired stable isometric immersion of the variable weighted Fermat-Frechet spherical multitree to the variable weighted Godel-Frechet multisimplex in $\mathbb{R}^{N}.$
\end{proof}

We focus on small perturbations $\epsilon_{ij}$ of a $\frac{(N-1)N}{2}-$tuple of positive real numbers $a_{ij},$ determining the edge lengths of $(N-1)$ incongruent spherical simplexes, such that: $a_{ij}, a_{ij}+\epsilon_{ij}\in DW_{\mathbb{S}_{\rho_{0}}^{N-1}} (\ell,s)$ and $s\le a_{ij}\le \ell,$
under the condition: \[\sum_{i,j=1,i<j}^{N} a_{ij}=\sum_{i,j=1,i<j}^{N} a_{ij}+\epsilon_{ij},\]
for $\|\epsilon_{ij}\|<<1$ and $\sum_{i,j=1, i<j}^{N}\epsilon_{ij}=0.$

By substituting $a_{ij}+\epsilon_{ij} \to a_{ij}$ in Theorem~\ref{stabletheorem1}, we derive the conditions to create a family of stable isometric immersions of spherical variable weighted Fermat-Frechet multitrees to Godel-Frechet multisimplexes depending on the isoperimetric perturbations $\epsilon_{ij}$ of the initial boundary simplex in $\mathbb{S}_{\rho_{0}}^{N-1}.$

\begin{theorem}[Stable isometric immersions of weighted Fermat-Frechet multitrees for isoperimetric deformations of the boundary $N-$simplex in $\mathbb{S}_{\rho_{0}}^{N-1}$  to $\mathbb{R}^{N}$ ] \label{stabletheorem2}

The following conditions for $B_{i}(t)$ provide a family of stable isometric immersions of variable weighted Fermat-Frechet multisimplexes in $\mathbb{S}_{\rho}^{N-1}$ to $\epsilon_{ij}$-Godel-Frechet simplexes in $\mathbb{R}^{N}:$

\begin{align}\label{construct312nsminus2rvarstab1}
\sum_{i=2}^{N}\frac{B_{i}(t)}{\sin K a_{01}\sin K a_{0i}}(-\cos Ka_{01}\cos Ka_{0i}+ \cos K (a_{1i}+\epsilon_{1i})-\nonumber \\ \sum_{i=1,{i\ne
j}}^{N}\frac{B_{i}(t)}{\sin Ka_{0j} \sin Ka_{0i}}(-\cos Ka_{0j}\cos Ka_{0i}+\cos K (a_{ji}+\epsilon_{ji})=0,
\end{align}

\begin{equation}\label{isopcondvariableweightsstab2}
\sum_{i=1}^{N}B_{i}(t)=1,
\end{equation}

\begin{equation}\label{sphericalvolsminus2rvarstab3}
\operatorname{Vol}(A_{1},A_{2},\cdots,A_{0}\cdots, A_{N})(\{\epsilon_{ij}\})<\operatorname{Vol}(C_{1},C_{2},\cdots,C_{N}))(\{\epsilon_{ij}\})
\end{equation}

where $C_{1},C_{2},\cdots C_{N}$ is a regular spherical $(N-1)$simplex with edge length $c_{ij}=\frac{\sum_{i,j}a_{ij}+\epsilon_{ij}}{\frac{(N-1)N}{2}}.$


for the $\frac{(N(N+1}{2}$tuple of edge lengths $\{a_{01},\cdots,a_{0N},\{a_{ij}\}\}\in DW_{\mathbb{R}^{N}}(\ell, s).$

\end{theorem}


\section{Gromov isometries for the Fermat-Steiner-Frechet solution in $\mathbb{H}^{N}$}
In this section, we introduce an embedding (inclusion map) of a Fermat-Steiner-Frechet multitree solution for a given $\frac{N(N+1}{2}$-tuple of edge lengths determining incongruent boundary $N$-simplexes to an associated family of Gromov isometries up to an additive constant for $N$-simplexes and ideal $N$-simplexes in $\mathbb{H}_{K}^{N}.$


We proceed by giving the definitions of an ideal $k$-simplex $\triangle_{i}^{k}\equiv A_{1}^{i}A_{2}^{i}\cdots A_{k+1}^{i}$ in $\mathbb{H}_{N}^{k}$ with constant curvature $K=-\epsilon^{2}<0,$ a class of geodesic trees enriched by some properties given by Gromov (intermediate geodesic Fermat-Steiner trees) and Gromov's (+) isometry (up to an additive constant).


The hyperbolic $N$ space $\mathbb{H}_{N}^{k}$ may be represented as the Poincare disc model or the half space model or the projective (Klein) model.

If we consider that $\mathbb{H}_{N}$ is represented by the Poincare disc model
\[\mathbb{H}^{N}\approx D^{N}=\{x\in \mathbb{R}^{N}| ||x||<1  \}\] with the Riemannian metric
\[ds^{2}=\frac{4}{1-r^2}\sum_{i=1}^{N}(dx_{i})^{2}, \ r^{2}\equiv \sum_{i=1}^{N}x_{i}^{2} \] (see in \cite{HaagerupMunhkholm:81}).

\begin{definition}{\cite[p.~1]{HaagerupMunhkholm:81}}\label{hyperboundary}
The hyperbolic boundary $\partial \mathbb{H}^{N}$ is the "sphere at infinity"
\[\partial \mathbb{H}^{N}=\{x\in \mathbb{R}^{N}| ||x||=1\}=\mathbb{S}^{N-1}.\]
\end{definition}


\begin{definition}{An ideal $k$-simplex $\triangle_{i}^{k},$ \cite[p.~233]{Gromov:87}}
An ideal $N$-simplex in $\mathbb{H}_{-\epsilon^{2}}^{N}$ is the convex hull of $N+1$ distinct points (vertices) \\$\{A_{1}^{id},A_{2}^{id},\cdots ,A_{N}^{id},A_{N+1}^{id}\}$ in the hyperbolic boundary $\partial \mathbb{H}_{-\epsilon^{2}}^{N}$ ("sphere at infinity").
\end{definition}
In \cite{Thurston:97}, Thurston showed that all ideal triangles in $\mathbb{H}_{-1}^{2}$ are isometric. In \cite{Milnor:94} Milnor derived that ideal hyperbolic tetrahedra are not mutually isometric, In \cite{Gromov:87}[pp.~233-234]Gromov assigned properties in a geodesic tree $ S\subset \Delta^{N}\subset \mathbb{H}^{N}$ and used in a canonical way an inclusion map $S\to \Delta^{N},$ to confirm that ideal $N-$simplexes are not mutually isometric for $N\ge 3.$

We extend the definitions of Steiner tree topologies given in \cite{GilbertPollak:68} and intermediate Fermat-Steiner tree topologies, which have been introduced in \cite{Zachos:17} in $\mathbb{R}^{3},$ for geodesic trees in $\mathbb{H}_{-\epsilon^{2}}^{N}.$

\begin{definition}[Geodesic tree topology]
A geodesic tree topology is a connection matrix specifying which pairs of points from the list\\ $\{A_{1}, A_{2},\cdots A_{N}, X_{1}, X_{2},\cdots X_{N-2}\} \in \mathbb{H}_{-\epsilon^{2}}^{N}$ have a connecting geodesic segment (edge).
\end{definition}

\begin{definition}[Degree of a vertex]
The degree of a vertex corresponds to the number of connections of the vertex with geodesic segments.
\end{definition}

\begin{definition}[Degree of an intermediate Fermat-Steiner point]
The degree of a (weighted) intermediate Fermat-Steiner point (vertex) $X$ with respect to a boundary simplex $A_{1}, A_{2},\cdots A_{N}\in \mathbb{H}_{-\epsilon^{2}}^{N}$ is greater or equal than $3 $ and less or equal than $N.$
\end{definition}

\begin{definition}[Fermat (geodesic) tree topology]
A Fermat tree topology is a tree topology  with a connection matrix $A_{1}, A_{2},\cdots A_{N}, X_{1},$ such that the boundary vertices $A_{1}, A_{2},\cdots A_{N}$ have degree $1$ and the Fermat-point $X_{1}$ has degree $N$.
\end{definition}

\begin{definition}[Fermat-Steiner (geodesic)) tree topology]
A (full) Fermat-Steiner tree topology is a tree topology  with a connection matrix $A_{1}, A_{2},\cdots A_{N}, X_{1},X_{2},\cdots X_{N-1},$ such that the boundary vertices of the simplex $A_{1}, A_{2},\cdots A_{N}$ have degree $1$ and the Fermat-Steiner points $X_{1},X_{2},\cdots X_{N-1}$ have degree $3.$
\end{definition}

\begin{definition}[Intermediate Fermat-Steiner (geodesic) tree topology in $\mathbb{H}_{k}^{N}$]
An intermediate Fermat-Steiner tree topology is a tree topology, which has $m\le N-2$ $X_{1},X_{2},\cdots X_{m}$ vertices inside the simplex $A_{1}, A_{2},\cdots A_{N} \in \mathbb{H}_{-\epsilon^{2}}^{N},$ such that the boundary vertices of the simplex $A_{1}, A_{2},\cdots A_{N}$ have degree $1$ and the intermediate Fermat-Steiner points $X_{1},X_{2},\cdots X_{m}$ have degree less than  $N.$
\end{definition}

The intermediate Fermat-Steiner tree $S$ is a union of some geodesic segments between the intermediate Fermat-Steiner vertices $X_{j}$ and some geodesic segments connecting each $A_{i}$ with some $X_{j}.$
In \cite{Gromov:87}, Gromov assigned in a canocical way an intermediate Fermat-Steiner geodesic tree for an ideal simplex  $A_{1}, A_{2},\cdots A_{N}$ some geometric properties \cite[(a),(b),(c),(d),(e), p.~233]{Gromov:87}.

\begin{example}
An intermediate Fermat-Steiner tree $S$ of a $4$-simplex $A_{1}, A_{2}, A_{3}, A_{4},A_{5}$ associated with an intermediate Fermat-Steiner tree topology is a collection of the geodesic segments $\{A_{1}X_{1},A_{4}X_{1},X_{1}X_{2},$\\ $A_{5}X_{2},A_{2}X_{2}, A_{3}X_{2}\}.$ The intermediate Fermat Steiner points $X_{1},$ $X_{2}$ have degrees $3$ and $4,$ respectively and the boundary vertices $\{A_{1}, A_{2}, A_{3},$\\$ A_{4},A_{5}\}$ have degree $1.$
\end{example}

\begin{definition}[Intermediate Fermat-Steiner Frechet multitree]\label{FSFtree}
An intermediate Fermat-Steiner Frechet multitree is a collection of intermediate Fermat-Steiner trees, which correspond to incongruent $N$-simplexes derived by a $\frac{N(N+1}{2}-$tuple of positive real numbers determining edge lengths.
\end{definition}

\begin{remark}
An intermediate Fermat-Steiner Frechet multitree $MS$ coincides with a Fermat-Frechet multitree, by setting only one vertex $X_{1}$ inside each derived simplex (Frechet multisimplex $F\Delta^{N}$ ) having degree $N.$
\end{remark}

\begin{definition}A Fermat-Frechet multispanning tree is a Fermat-Frechet multitree with zero interior vertices ($m=0$) at each simplex, which belongs to the Frechet multisimplex.
\end{definition}

In \cite[6.2, pp.~157-158]{Gromov:87}, Gromov constructed a geodesic tree ($\log_{2}N$  spanning geodesic tree)  $\tau ]subset X,$ for a finite subset $\{V=V_{1},V_{2},\cdots, V_{N}\}$ of a geodesic $\delta$ hyperbolic space $\mathbb{X}$ with following properties:

(1) $\tau$ is a union of at most $N-1$ geodesic segments in $\mathbb{X}$ ,\\

(2) the set of extremal points of $\tau$ equals $V,$\\

(3) Every two points $V_{i},$ $V_{j}$ in $V$ can be joined by a broken geodesic g having $k\le 1+2 log_{2}N$ segments, such that:

\[length g \le |V_{i}-V_{j}|+ C \delta (log_{2}(N))^{2},\] for $C\le 100.$

In \cite{Bowditch:91}, \cite{Bowditch:06}, Bowditch obtained another tree-like structure of a $\delta $ hyperbolic space $\mathbb{X}$ in the sense of Gromov, by determining an isometry of a spanning tree up to an additive constant, which depend on $N$ and a constant $\delta,$ which lead to an isometric embedding of "logarithmic ($\log N$) spanning trees".

\begin{lemma} {Isometric embedding of a spanning tree up to an additive constant,  \cite[Proposition~6.7,pp.~49-51]{Bowditch:06}}
There is a function $h:\mathbb{N}\to [0,\infty)$ such that if $F\subset \mathbb{X}$ with $|F|=\mathbb{N},$ then there is a (spanning) tree $\tau,$ such that for all $X,Y\in F,$ $d_{\tau}(X,Y)\le XY+\delta h(N),$ ($d_{\tau}$ is distance measured in $\tau $ and all the edges of $\tau $ are geodesic segments).
\end{lemma}

\begin{lemma}{Isometric embedding of a  spanning tree up to an additive logarithmic constant, \cite[Proposition~7.3.1, Theorem~7.6.1]{Bowditch:91}}
There is a function $f:\mathbb{N}\to \mathbb{R},$ such that the following holds.
Suppose ($\mathbb{X},d)$ is a $\delta$ hyperbolic geodesic space and the vertex set $V\subset \mathbb{X}$ is a set of $(N+1)$points. Then there is an immersed spanning tree $\tau$ for $V$ in $\mathbb{X},$ such that for any $X,Y\in V,$
we have
\[\rho_{\tau}(X,Y))\le d(X,Y)+hF(N),\] where $F(N)=O(log N).$
\end{lemma}

We proceed by obtaining an embedding of Fermat-Steiner Frechet multitree to $\Delta^{N}$, taking into consideration Gromov's observation that the regular ideal simplex having the maximal volume corresponds to a tree $S$ having a maximal length metric among the Fermat-Steiner trees, which consists of the union of $N+1$ rays joining the ideal points $A_{i}^{\circ}$ with only a single point $X_{1}\in \mathbb{H}^{N}$      (\cite[Property (d),p.~233]{Gromov:87}).

We consider an example of a tree $S$ for a boundary tetrahedron in $\mathbb{R}^{3}.$ We show that we can perturb the maximal length (metric) of the unweighted Fermat- Steiner trees, by assigning properly some weights at each vertex of the tetrahedron, which yields a  maximal weighted length metric of the Fermat tree greater than the maximal unweighted length metric of the weighted Fermat tree having the same geometric structure (tree topology).

\begin{example}
Let $A_{1}A_{2}A_{3}A_{4}$ be a tetrahedron and $X_{1}$ be the corresponding Fermat point in $\mathbb{R}^{3}.$
One can rotate by a suitable angle (twist angle), such that $A_{1}A_{2}A_{3}A_{4}$ and $X_{1}$ lie on the same plane. Thus, $X_{1}$ is the intersection of the diagonals $A_{1}A_{3}$ and $A_{2}A_{4}.$ One can assign the following weights $B_{i}$ (positive real numbers) at each vertex $A_{i}$ (Fig.~\ref{figg1}).
The length of the unweighted ($B_{i}=1$) Fermat tree is given by
\[length g = A_{1}A_{3}+A_{2}A_{4}.\]
We can perturb the length metric of the Fermat tree, by assigning the following weights at each $A_{i},$ such that the weighted Fermat point $X_{1}$ remains the same:
\[B_{1}=B_{3}=1+\epsilon,\ B_{2}=B_{4}=1-\epsilon.\]
\[length g_{p} = (1+\epsilon)A_{1}A_{3}+(1-\epsilon)A_{2}A_{4}.\]

Hence, we obtain the inequality for $\epsilon >0$ and $a_{13}>a_{24}$ or $\epsilon <0$ and $a_{13}<a_{24}:$
\[length g_{p}-length g=\epsilon (a_{13}-a_{24})>0.\]

\begin{figure}
\centering
\begin{center}
\includegraphics[scale=0.90]{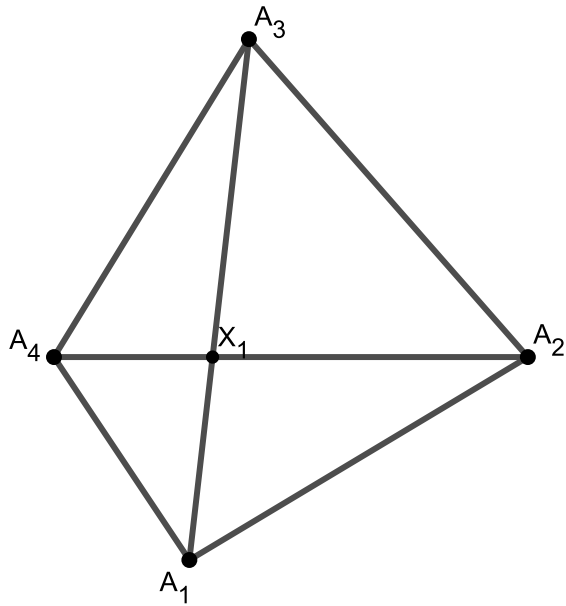}
\caption{ Perturbing the Fermat length metric for $A_{1}A_{2}A_{3}A_{4}$ with weights $B_{1}=B_{3}=1+\epsilon,\ B_{2}=B_{4}=1-\epsilon$ } \label{figg1}
\end{center}
\end{figure}

\end{example}


\begin{theorem}\label{thmembedmultitree}
The embedding of an intermediate Fermat-Steiner Frechet multitree $MS\to F\Delta^{N}$ is a collection of isometries up to an additive constant, where the implied constants only depends on $N$ and $\epsilon.$
\end{theorem}

\begin{proof}

A. A Proof using weights

We can assign weights $B\{_{1}, B_{2},\cdots B_{N+1}\}$ to each vertex of the simplex $A_{1}A_{2}\cdots A_{N+1}\in \mathbb{H}_{-\epsilon^{2}}^{N},$ to perturb the corresponding maximal length metric of the Fermat-tree, which gives an isometry up to an additive $\log_{2}N$ or $\log N$ constant due to Gromov, Bowditch for a given $\frac{N(N+1}{2}$tuple of edge lengths determining a simplex by using the solution domain of Dekster-Wilker DW($\mathbb{H}_{\epsilon^{2}}^{N}$). Thus, the collection of the isometries up to an additive constant depend on $\epsilon$ and $N$ and gives a "multi-"inclusion map of intermediate Fermat-Steiner-Frechet multitrees $MS$ to a Frechet multisimplex $F\Delta^{N}.$

B. Proof without using weights

We use the following result for the length metric of minimal Steiner trees and associated minimal spanning trees, which has been proved by Ivanov Tuzhilin and Cieslik for manifolds (\cite{IvanovTuzhilinCieslik:03}, \cite{IvanovTuzhilin:01a}).
The length of the Fermat-Steiner tree with intermediate vertices is less than length of the spanning tree with zero intermediate vertices.
Therefore, by applying Gromov-Bowditch constructions, we get
\[length(Fermat\ Steiner\ tree)<length{Spanning \ tree}\le |V_{i}-V_{j}|+ C \delta (log_{2}(N))^{2},\]
for $C\le 100.$

\end{proof}

We define a dynamic intermediate Fermat-Steiner Frechet multitree for some perturbed hyperbolic Frechet multsimplexes and we consider emmbeddings of dynamic intermediate Fermat-Steiner Frechet multitrees to almost regular Frechet multisimplex in terms of Dekster's lower bound estimate for the generalized Santalo's thickness of simplexes in $\mathbb{H}_{-\epsilon^{2}}^{N}$ (\cite[Theorem~2 (2.3),pp.51~52]{Dekster:97}).


\begin{definition}[A dynamic intermediate Fermat-Steiner Frechet multitree]\label{FSFtree}
A dynamic intermediate Fermat-Steiner-Frechet multitree is a union of Fermat-Steiner-Frechet multitrees, which are derived by a given $\frac{N(N+1)}{2}-$tuple of positive real numbers determined by the edge lengths $\ell_{i}$ of a Frechet multisimplex in $\mathbb{H}_{k}^{N}$ perturbed by a real number $\epsilon_{i}:$ $\ell_{i}^{\prime}=\ell_{i}+\epsilon_{i}:$
\[\sum_{i=1}^{\frac{N(N+1)}{2}}\ell_{i}^{\prime}=\sum_{i=1}^{\frac{N(N+1)}{2}}\ell_{i}^{\prime},\]
where $\sum_{i=1}^{\frac{N(N+1)}{2}}\epsilon_{i}=0$ and $\ell_{i}^{\prime}$ are the edge lengths of a Frechet
simplex obtained by the Dekster Wilker solution domain $DW( \mathbb{H}_{k}^{N} ).$

\end{definition}

We note that Dekster (\cite[p~52, (1.14), Example,pp.~53-54]{Dekster:97}) extended Gromov's thinness of geodesic triangles in $\delta=\frac{1}{2\epsilon}\log 3$ hyperbolic spaces (\cite[6.3, Lemma~6.3.A]{Gromov:87},\cite[Example~4,p.~54]{Bowditch:06}), by obtaining a generalization of Santalo's thickness from $\mathbb{H}^{2}$ to $\mathbb{H}^{N}$ and distinguise the thickness of simplexes of a compact convex set $C$ (simplex) )$t_{g}(C)$ to the general thickness $t_{g}(C)$ and the normal thickness $t_{N},$ which yields $t_{N}(C)\ge t_{g}(C).$

\begin{lemma}{\cite[Theorem~2, (2.3), pp.~54-55]{Dekster:97}}
Let $\Delta^{N}$ be an $N-$simplez in $\mathbb{H}^{N}$ with edges $a_{ij}\in [s,\ell]\subset [\lambda_{N}(\ell),\ell$ where $\ell>0$ and $\ell_{N}(\ell)$ is given by
\[\cosh \lambda_{N}(\ell)=\sqrt{\cosh^{2}\ell -2 f_{N}(\cosh \ell -1)[1+\frac{N}{N+1}(\cosh \ell)-1,] }\]

where

\[ f_{N})= \left\{
\begin{array}{ll}
      \frac{2}{N+1} & for\ odd\ N, \\
      \frac{2(N+1)}{N(N+2)} & for\ even\ N

\end{array}
\right.
\]

Then the general thickness $t_{g}$ of $\Delta^{N}$ satisfies:
\[t_{g}\ge \cosh^{-1}(1+\frac{\cosh s-\cosh \lambda_{N}(\ell)}{\cosh^{2}\ell}).\]
\end{lemma}

\begin{theorem}[Isometric embedding of a dynamic Fermat-Steiner-Frechet multitree to a Frechet almost regular multisimplex]\label{Gromovdynamicembedding}
The embedding of a dynamic intermediate Fermat-Steiner Frechet multitree $MS\to F\Delta^{N},$ $F\Delta_{\epsilon_{i}}^{N}$ is a family of almost regular multiFrechet simplexes for $s=\ell -\max \{|\epsilon_{1}|,\cdots |\epsilon_{i}|\}$ very close to $\ell$, which gives a collection of isometries up to an additive constant, where the implied constants only depends on $N$ and $\epsilon,$ such that
$t_{g}(F\Delta ^{N})\ge \frac{1}{2\epsilon}\log 3$ for $s \le a_{ij}\le \ell$ and $(s,\ell)\in DW_{\mathbb{H}_{-\epsilon^{2}}^{N}}.$
\end{theorem}

\begin{proof}
Let $A_{1}A_{2}\cdots A_{N+1}$ be an almost regular simplex for $s$ very close to $\ell$ and  $X_{1}$ is the corresponding Fermat point. By taking $X_{1}$ as a reference point, the rays $X_{1}A_{i}$ isuuing from $X_{1}$ intersect $\partial \mathbb{H}_{-\epsilon^{2}}^{N}$ at $A_{i}^{\circ},$ for $i=1,2,\cdots, N,$ which yields that $A_{1}^{\circ}A_{2}^{\circ}\cdots A_{N}^{\circ}$ is an ideal almost regular $N$simplex. By using Gromov's observation that we can assign a Fermat tree having a maximal length metric to each $A_{1}^{\circ}A_{2}^{\circ}\cdots A_{N}^{\circ},$ which has an almost maximal volume and by applying  Theorem~\ref{thmembedmultitree}, we can embed isometrically up to an additive constant to each member ($N$ almost regular simplex) of multiFrechetsimplexes $F\Delta_{\epsilon_{i}}^{N},$ for  $s=\ell -\max \{|\epsilon_{1}|,\cdots |\epsilon_{i}|\}$ very close to $\ell:$ $s\le a_{ij} \le \ell.$

The Fermat point of $A_{1}A_{2}\cdots A_{N+1}$ and $A_{1}^{\circ}A_{2}^{\circ}\cdots A_{N+1}^{\circ}$ remains the same, because \[\sum_{i}^{N+1}\frac{\exp_{X_{1}}^{-1}X_{1}A_{i}}{|X_{1}A_{i}|}=\sum_{i}^{N+1}\frac{\exp_{X_{1}}^{-1}X_{1}A_{i}^{\circ}}{|X_{1}A_{i}^{\circ}|}=0.\]

\end{proof}


Let $S$ be an intermediate Fermat-Steiner tree with two intermediate Fermat-Steiner points $X_{1}$ and $X_{2}$ inside the ideal $N$ simplex  $A_{1}^{\circ}\cdots A_{N+1}^{\circ}$ in $\mathbb{H}_{k}^{N}$ having degree $p$ and $(N+2)-p$ and the boundary ideal vertices $A_{i}^{\circ}$ having degree one, for $3\le p\le N-2.$

\begin{lemma}{Convergence of $\Delta^{3}$ \cite[Example,~p.~234]{Gromov:87}}\label{conv3}
If $X_{1}X_{2}\to \infty,$ for $N=3,$ $p=3,$ then $\Delta^{3}$ converges exponentially fast converges to the union of two triangles spanned by the triples $\{A_{1}^{\circ},A_{2}^{\circ},\frac{X_{1}+X_{2}}{2}\}$ and $\{A_{3}^{\circ},A_{4}^{\circ},\frac{X_{1}+X_{2}}{2}\}.$
\end{lemma}

\begin{lemma}[Extension of convergence of an $N$simplex to a $p$simplex and $N+3-p$ simplex ]\label{convN}
If $X_{1}X_{2}\to \infty,$ then $\Delta^{N}$ exponentially fast converges to the union of two hyperbolic simplexes spanned by the
$p$-tuple $\{A_{1}^{\circ},A_{2}^{\circ},\cdots,A_{p-1}^{\circ},\frac{X_{1}+X_{2}}{2}\}$ and the\\ $(N+3-p)$-tuple
$\{A_{p}^{\circ},A_{p+1}^{\circ},\cdots,A_{N+2-p}^{\circ},\frac{X_{1}+X_{2}}{2}\},$ respectively.
\end{lemma}

In \cite[Example,p.~234]{Gromov:87}, Gromov also observed that if the basic invariant of $\Delta^{3}$  $X_{1}X_{2}<\infty,$ then the geometry of the ideal simplex $\Delta^{3}$ is close to that of the regular ideal simplex.

In Nature, an evolutionary tree with respect to a boundary closed polyhedron in $\mathbb{R}^{3}$ tends to maximize the volume of the closed polyhedron formed by the boundary vertices with a mimimum communication via minimum transfer of mass along the branches. The evolutionary tree reduces its length by placing intermediate Fermat-Steiner points inside the boundary polyhedron.

\begin{definition}[Degree of intelligence]
We call degree of intelligence of an evolutionary tree in $\mathbb{H}_{K}^{N}$ the number of intermediate Fermat Steiner points, which correspond to an intermediate Fermat Steiner tree of an $N$-simplex in $\mathbb{H}_{K}^{N}.$
\end{definition}

\begin{definition}[Maximum degree of intelligence of an evolutionary tree]
The maximum degree of intelligence of an evolutionary tree, which correspond to an intermediate Fermat Steiner tree of an $N$-simplex in $\mathbb{H}_{K}^{N}$ is $N-2.$
\end{definition}

Consider two trees $T$ and $T^{\prime}$ having the same set of boundary vertices in $\mathbb{H}_{k}^{2}.$

\begin{definition}[Intelligent trees]
 The tree $T$ is more intelligent than then tree $T^{\prime}$ if the degree of intelligence of $T$ is greater than
 the degree of intelligence of $T^{\prime}.$
\end{definition}

We continue with Gromov's definition of thin triangles in $\delta$ hyperbolic metric spaces, which help us to view  $\delta$ hyperbolic metric spaces as metric trees with a prescribed thickness (\cite{Gromov:87},\cite[1.1 Fefinition, Fig.~H.1,p.~399]{BridsonHaefliger:99}).

\begin{definition}{Thin triangle,\cite[p.~183]{Gromov:87}\cite[1.1 Fefinition, Fig.~H.1,p.~399]{BridsonHaefliger:99}}
A geodesic triangle in a metric space is said to be $\delta$ thin if each of its sides is contained in the $\delta$ neighborhood of the union of the other two sides.
\end{definition}
A metric (geodesic) tree with zero thickness is an $\mathbb{R}$ tree (Zero hyperbolic geodesic space)

\begin{definition}[Degree of intelligence of a $\delta$ metric tree]
The degree of intelligence of a $\delta$ metric tree embedded in an $N$ simplex is the degree of intelligence
of the corresponding $zero$ ($\delta=0$) metric tree (intermediate Fermat-Steiner tree) with respect to the same boundary $N$-simplex.
\end{definition}

Consider an intermediate Fermat-Steiner tree having two degrees of intelligence $X_{1},$ $X_{2}$ with corresponding degrees $p$ and $N+2-p,$ respectively inside the ideal simplex $A_{1}^{\circ}A_{2}^{\circ}\cdots A_{N}^{\circ}\in \mathbb{H}_{-\epsilon^{2}}^{N}.$
\begin{theorem}[Reduction of the degree of intelligence of a metric $\delta$ tree]\label{reductionintelligence}
If the hyperbolic distance $X_{1}X_{2}\to infinity,$ then the degree of intelligence of the $\frac{1}{2\epsilon^{2}}\log 3 $ metric tree is reduced by one degree from two to one.
\end{theorem}

\begin{proof}
By applying Lemma~\ref{convN} and taking into account $X_{1}X_{2}\to \infty,$ $\Delta^{N}$ exponentially fast converges to the union of two hyperbolic simplexes spanned by the
$p-tuple$ $\{A_{1}^{\circ},A_{2}^{\circ},\cdots,A_{p-1}^{\circ},\frac{X_{1}+X_{2}}{2}\}$ and the $N+3-p$tuple
$\{A_{p}^{\circ},A_{p+1}^{\circ},\cdots,A_{N+2-p}^{\circ},\frac{X_{1}+X_{2}}{2}\},$ respectively.

Hence, one may consider the vertex $\frac{X_{1}+X_{2}}{2}$ as the weighted Fermat point with respect to the union of the two hyperbolic simplexes, which minimize the objective function (zero Fermat (metric) geodesic tree) \[w_{1 }A_{1}^{\circ }\frac{X_{1}+X_{2}}{2}+w_{2}A_{2}^{\circ}\frac{X_{1}+X_{2}}{2}+w_{p-1}A_{p-1}^{\circ}\frac{X_{1}+X_{2}}{2}+\cdots+w_{N+2-p}A_{N+2-p}^{\circ}\frac{X_{1}+X_{2}}{2}. \]
Taking into account the corresponding $\delta=\frac{1}{\epsilon^{2}}\log 3$ we may consider a $k \delta$ hyperbolic metric tree (\cite[Lemma,p.~183]{Gromov:87}) as a union of the $\delta$ branches $A_{i}^{\circ}\frac{X_{1}+X_{2}}{2},$ for $i=1,2,\cdots, N+1,$ for a proper $N+1$ tuple of weights
(positive real numbers), we obtain a metric tree structure having one degree of intelligence (one Fermat point).
\end{proof}


\section{A variational approach to the weighted Fermat-Frechet problem for a given sextuple of edge lengths determining a $3-$simplex in the three-dimensional $K-$Space}
In this section, we obtain a method to differentiate the length of a geodesic arc
with respect to a geodesic arc in the 3 $K-$Space, which is different than the method derived by Schlafli (\cite{Schlafli:53})and Luo (\cite{LuoFeng:08}). This new variational technique can be applied to solve the weighted Fermat-Frechet problem for boundary incongruent tetrahedra in the 3 $K-$Space and determine the corresponding weighted Fermat trees for each derived tetrahedron, such that their edges belongs to the Dekster Wilker spherical domain $DW_{\mathbb{S}_{\frac{1}{\sqrt{K}}}^{3}}(\ell, s)$ or hyperbolic domain $DW_{\mathbb{H}_{K}^{3}}(\ell, s).$

Let $A_{1}A_{2}A_{3}A_{4}$ be a tetrahedron in the 3$K-$Space and $A_{0}$ be the corresponding weighted Fermat point inside the tetrahedron.

We denote by $\alpha$ the dihedral angle formed by the planes $\triangle A_{1}A_{2}A_{3}$ and $\triangle A_{1}A_{2}A_{0}$ and by $\alpha_{g_{4}}$ the dihedral angle formed by the planes $\triangle A_{1}A_{2}A_{3}$ and $\triangle A_{1}A_{2}A_{4},$ by $h_{0,12}$ the height of $\triangle A_{1}A_{2}A_{3}$ from $A_{0}$ to $A_{1}A_{2},$

by $A_{0,12}$ the projection of $A_{0}$ to $A_{1}A_{2}$   by $A_{0,123}$ the projection of $A_{0}$ to the $K-$plane defined by $\triangle A_{1}A_{2}A_{3}$ by $h_{0,123}$ the length of $A_{0}A_{123},$ by $x_{i}$ the length of $A_{i}A_{123},$ by $d$ the length of $A_{0,12}A_{0,123}$ and by $l$ the length of $A_{0,12}A_{2},$ for $i=1,2,3.$

We set $\alpha_{ijk}\equiv \angle A_{i}A_{j}A_{k},$ $\beta \equiv \angle A_{1}A_{2}A_{0,123},$ $h_{0,12}\equiv A_{0}A_{0,12},$ $h_{0,123}\equiv A_{0}A_{0,123}$ and $l\equiv A_{2}A_{0,12}.$

 for $i,j,k=0,1,2,3$

\begin{theorem}\label{a04a01a02a03}
The geodesic edges $a_{03}$ and $a_{04}$ can be expressed as functions of $a_{01},a_{02},\alpha:$


\begin{eqnarray}\label{Kplane24}
\cos\kappa a_{03}=\cos\kappa a_{02}\cos\kappa a_{23}+\nonumber\\
+\sin\kappa a_{23}\cos\kappa h_{0,12}\cos\alpha_{123}\sin\kappa l(a_{01},a_{02})+\nonumber\\
+\sin\kappa a_{23}\sin\alpha_{123}\cos\alpha
\end{eqnarray}

\begin{eqnarray}\label{Kplane25}
\cos\kappa a_{04}=\cos\kappa a_{02}\cos\kappa a_{24}+\nonumber\\
+\sin\kappa a_{24}\cos\kappa h_{0,12}\cos\alpha_{124}\sin\kappa l(a_{01},a_{02})+\nonumber\\
+\sin\kappa a_{24} \sin\alpha_{124}\cos(\alpha_{g_{4}}-\alpha)
\end{eqnarray}

\end{theorem}

\begin{proof}
From the cosine law in $\triangle A_{1}A_{2}A_{3}\in K$-plane, we get:
\begin{eqnarray}\label{Kplane1}
\cos\alpha_{102}=\frac{\cos \kappa a_{12} - \cos \kappa a_{01}\cos\kappa a_{02}}{\sin \kappa a_{01} \sin\kappa a_{02}}
\end{eqnarray}

From the sine law in $\triangle A_{2}A_{0}A_{0,12},$ we get:

\begin{eqnarray}\label{Kplane21}
\sin\alpha_{120}=\frac{\sin\kappa h_{0,12}}{\sin \kappa a_{02}}
\end{eqnarray}

By substituting (\ref{Kplane21})in the sine law of $\triangle A_{1}A_{2}A_{0},$ we derive:

\begin{eqnarray}\label{Kplane22}
\frac{\sin \kappa a_{12}}{ \sin \alpha_{102}}=\frac{\sin \kappa a_{1}}{\sin \alpha_{120}}=\frac{\sin \kappa a_{01}\sin \kappa a_{02}}{\sin \kappa h_{0,12}}
\end{eqnarray}

or

\begin{eqnarray}\label{Kplane3}
\sin\kappa h_{0,12}=\frac{\sin \kappa a_{01}\sin \kappa a_{02}}{\sin \kappa a_{12}} \sin \alpha_{102}.
\end{eqnarray}

By substituting (\ref{Kplane1}) in (\ref{Kplane3}), we obtain:

\begin{eqnarray}\label{Kplane456}
\sin\kappa h_{0,12}=\frac{\sin \kappa a_{01}\sin \kappa a_{02}}{\sin \kappa a_{12}} \sqrt{1-(\frac{\cos \kappa a_{12} - \cos \kappa a_{01}\cos\kappa a_{02}}{\sin \kappa a_{01} \sin\kappa a_{02}})^{2}}
\end{eqnarray}

From the cosine law in $\triangle A_{0}A_{0,123}A_{3},$ we get:
\begin{eqnarray}\label{Kplane7}
\cos \kappa a_{03}=\cos\kappa h_{0,123}\cos \kappa x_{3}
\end{eqnarray}

From the cosine law in $\triangle A_{0,123}A_{3}A_{2},$ we have:

\begin{eqnarray}\label{Kplane8}
\cos \kappa x_{3}=\cos\kappa x_{2} \cos \kappa a_{23}+\sin\kappa x_{2} \sin \kappa a_{23}\cos(\alpha_{123}-\beta)
\end{eqnarray}

By substituting (\ref{Kplane8}) in (\ref{Kplane7}), we get:

\begin{eqnarray}\label{Kplane9}
\cos\kappa a_{03}=\cos\kappa x_{2}\cos\kappa h_{0,123}\cos\kappa a_{23}+\nonumber\\
+\sin\kappa x_{2}\cos\kappa h_{0,123}\sin\kappa a_{23}\cos(\alpha_{123}-\beta)
\end{eqnarray}

From the cosine law in $\triangle A_{0,123}A_{2}A_{0},$ $\triangle A_{0}A_{0,12}A_{2},$ we get respectively:

\begin{equation}\label{Kplnae9bis1}
\cos\kappa a_{2}= \cos\kappa x_{2}\cos\kappa h_{0,123}
\end{equation}

and

\begin{equation}\label{Kplnae9bis2}
\cos\kappa h_{0,123}=\cos\kappa \frac{\cos\kappa h_{0,12}}{\cos\kappa d}.
\end{equation}

By substituting (\ref{Kplnae9bis1}) and (\ref{Kplnae9bis2}) in (\ref{Kplane9}), we have:

\begin{eqnarray}\label{Kplane10}
\cos \kappa a_{03}=\cos\kappa a_{02}\cos\kappa a_{23}+\nonumber\\ + \sin\kappa a_{23}\sin\kappa x_{2}\frac{\cos\kappa h_{0,12}}{\cos\kappa d}\cos\alpha_{123}\cos\beta+\nonumber\\+ \sin\kappa a_{23}\sin\kappa x_{2}\frac{\cos\kappa h_{0,12}}{\cos\kappa d}\sin\alpha_{123}\sin\beta
\end{eqnarray}

From the sine law and cosine law in $\triangle A_{0,123}A_{0,12}A_{2},$ we get:

\begin{eqnarray}\label{Kplane111}
\sin\kappa x_{2}=\frac{\sin\kappa d}{\sin\beta}
\end{eqnarray}

\begin{eqnarray}\label{Kplane112}
\cos\kappa x_{2}=\cos\kappa d \cos\kappa l.
\end{eqnarray}

By substituting (\ref{Kplane111}) and (\ref{Kplane112}) in (\ref{Kplane10}), we have:

\begin{eqnarray}\label{Kplane12}
\cos\kappa a_{03}=\cos\kappa a_{02}\cos\kappa a_{23}+\nonumber\\
+\sin\kappa a_{23}\tan\kappa d \cos\kappa h_{0,12}\cos\alpha_{123}\cos\beta \sin\kappa x_{2}+\nonumber\\
+\sin\kappa a_{23}\tan\kappa d \cos\kappa h_{0,12}\sin\alpha_{123}\sin\beta \sin\kappa x_{2}
\end{eqnarray}

From the cosine law and sine law in $\triangle A_{0,123}A_{0,12}A_{2},$ we get:

\begin{eqnarray}\label{Kplane121}
\cos\beta=\frac{\cos\kappa d -\cos \kappa l \cos\kappa x_{2}}{\sin \kappa l \sin\kappa x_{2}}
\end{eqnarray}

\begin{eqnarray}\label{Kplane13}
\sin\beta=\frac{\sin\kappa d}{\sin \kappa x_{2}}
\end{eqnarray}

\begin{eqnarray}\label{Kplane14}
\cos\kappa l=\frac{\cos\kappa a_{02}}{\cos\kappa h_{0,12}}
\end{eqnarray}

Taking into account that $h_{012}=h_{012}(a_{01},a_{02}),$ (\ref{Kplane14}) yields:

\[l=l(a_{01},a_{02})\]

By substituting (\ref{Kplane121}) and (\ref{Kplane13}) in (\ref{Kplane12}), we have:

\begin{eqnarray}\label{Kplane1567}
\cos\kappa a_{03}=\cos\kappa a_{02}\cos\kappa a_{23}+\nonumber\\
+\sin\kappa a_{23}\cos\kappa h_{0,12}\cos\alpha_{123}\sin\kappa l+\nonumber\\
+\sin\kappa a_{23}\cos\kappa h_{0,12}\sin\alpha_{123}\tan \kappa d
\end{eqnarray}

From the cosine law and sine law in $\triangle A_{0}A_{0,12}A_{0,123},$ we get:

\begin{eqnarray}\label{Kplane1819}
\cos\kappa h_{0,123}=\frac{\cos\kappa h_{0,12}}{\cos\kappa d}
\end{eqnarray}

\begin{eqnarray}\label{Kplane20}
\sin\kappa h_{0,123}=\frac{\sin\kappa h_{0,12}}{\sin\alpha}
\end{eqnarray}

\begin{eqnarray}\label{Kplane21}
\cos\kappa h_{0,123}=\cos\kappa d \cos\kappa h_{0,12}+\sin\kappa d \sin\kappa h_{0,12}\cos\alpha
\end{eqnarray}

Substituting (\ref{Kplane1819}), (\ref{Kplane20}) in (\ref{Kplane21}) yields:

\begin{eqnarray}\label{Kplane23}
\tan\kappa d=\tan\kappa h_{0,12} \cos\alpha
\end{eqnarray}

By substituting (\ref{Kplane23}) in (\ref{Kplane1567}), we obtain (\ref{Kplane24}).

By working cyclically and changing the index $3\to 4,$ we derive (\ref{Kplane25}).

\end{proof}

\begin{proposition}\label{propa04a01a02a03}
\begin{eqnarray}\label{Kplane24bis}
\kappa a_{04}=\kappa a_{04}(a_{01},a_{02},a_{03}; a_{12},a_{13},a_{14},a_{21},a_{23},a_{24},a_{34},\kappa).
\end{eqnarray}
\end{proposition}
\begin{proof}
By solving (\ref{Kplane24}) with respect to $\cos\alpha$ and substituting this variable in (\ref{Kplane25}) and by replacing $\alpha_{123},$ $\alpha_{124},$ from the cosine law and the sine law in $\triangle A_{1}A_{2}A_{4}$ and $A_{1}A_{2}A_{3}$ in (\ref{Kplane25}) we get (\ref{Kplane24bis}).
\end{proof}

\begin{theorem}[The weighted Fermat-Torricelli Frechet solution]\label{thm1vartetrahedron}

The following equations depending on the variable edge lengths $a_{01},a_{02}, a_{03}$ the constant sectional curvature $K$ and the edge lengths $\{a_{ij}\}$ provide a necessary condition for the location of the weighted floating Fermat-Torricelli trees $\{a_{01},a_{02},a_{03},a_{04}\},$ which belong to the weighted Fermat-Frechet (multitree solution in the 3$K$-Space:
\begin{align*}
\sum_{i=2}^{4}\frac{B_{i}}{\sinh K a_{01}\sinh K a_{0i}}(\cosh Ka_{01}\cosh Ka_{0i}-\cosh Ka_{1i})-\nonumber \\ \sum_{i=1,{i\ne
j}}^{4}\frac{B_{i}}{\sinh Ka_{0j} \sinh Ka_{0i}}(\cosh Ka_{0j}\cosh Ka_{0i}-\cosh Ka_{ji})=0,
\end{align*}

\begin{equation}\label{spherihyperbolicvolnbis}
\operatorname{Vol}(A_{0},A_{i},A_{j},A_{k})< \operatorname{Vol}(C_{1}C_{2},C_{3},C_{4}).
\end{equation}

for $i,j,k=1,2,3,4, i\ne j\ne k.$

\end{theorem}

\begin{proof}
It is a direct consequence of Proposition~\ref{propa04a01a02a03} and Theorems~\ref{floatmultitreehn},~\ref{floatmultitreesn} for $N=4.$
\end{proof}


In \cite[Problem~5, Theorem~3]{Zachos:21}, we study the weighted Fermat-Steiner problem for a boundary tetrahedron $A_{1}A_{2}A_{3}A_{4}$ in $\mathbb{R}^{3}.$
The weighted Fermat-Steiner solution is a tree having two weighted Fermat (Steiner) points $A_{0}$ and $A_{0}^{\prime}$ inside $A_{1}A_{2}A_{3}A_{4}.$ The weighted Fermat Steiner tree topology consists of the branches (line segments) $\{A_{1}A_{0},A_{2}A_{0},A_{0}A_{0}^{\prime},A_{0}^{\prime}A_{3},A_{0}^{\prime}A_{4}\},$ with
corresponding weights\\ $\{B_{1},B_{2},\frac{B_{0}+B_{0}^{\prime}}{2},B_{3},B_{4}\}.$ We call degree of intelligence of a a weighted Fermat-Steiner network with respect to a boundary $N$-simplex in $\mathbb{R}^{N}$ the number of weighted Fermat-Steiner points it possesses inside the simplex. Thus, a full weighted Fermat-Steiner network for boundary tetrahedra having two weighted Fermat-Steiner points has two degrees of intelligence.

In \cite[Figures~8,~9,pp.~18-19]{Alexandroff:61}, P. Alexandrov makes an elegant exposition of algebraic complexes by considering them as higher dimensional generalization of ordinary directed polygonal paths, taking into account that a line which is traversed twice in opposite directions does not count. By applying this technique to a weighted Fermat-Steiner network for boundary tetrahedra, we may isolate the two degrees of intelligence in two non-intersecting triangles $\triangle A_{1}A_{2}A_{0}$ and $\triangle A_{1}A_{2}A_{0}^{\prime}.$

We consider the following two networks:\\

$\bullet$ a weighted Fermat-Steiner tree for $A_{1}A_{2}A_{3}A_{4},$ such that:
\[f(A_{0},A_{0}^{\prime})=B_{1}A_{1}A_{0}+ B_{2}A_{2}A_{0}+B_{3}A_{3}A_{0}^{\prime}+B_{4}A_{4}A_{0}^{\prime}+\frac{B_{0}+B_{0}^{\prime}}{2}A_{0}A_{0}^{\prime}\to min\]

$\bullet$ the two broken lines formed by the boundaries $\triangle A_{1}A_{2}A_{0}$ and $A_{3}A_{4}A_{0}^{\prime}$
connected by their minimum distance $A_{0}A_{0}^{\prime}$

By using the orientation \[A_{0}\to A_{1}\to A_{2} \to A_{0}\to A_{0}^{\prime}\to A_{3}\to A_{4}\to A_{0}^{\prime}\to A_{0},\]

we obtain the following proposition, which deals with isolated intelligence of the boundaries $\triangle A_{1}A_{2}A_{0}$ and $A_{3}A_{4}A_{0}^{\prime}.$

\begin{proposition}[A directed weighted Fermat-Steiner Frechet isolated multitree for Frechet tetrahedra in $\mathbb{R}^{3}$ in the sense of P. Alexandrov]

If the orientation \[A_{0}\to A_{1}\to A_{2} \to A_{0}\to A_{0}^{\prime}\to A_{3}\to A_{4}\to A_{0}^{\prime}\to A_{0},\] for the boundary triangles $\triangle A_{1}A_{2}A_{0}$ and $A_{3}A_{4}A_{0}^{\prime}$ occurs for incogruent boundary tetrahedra derived by a sextuple of positive real numbers under the conditions of Blumenthal, Herzog, Dekster-Wilker, such that
$\{A_{1}A_{0},A_{2}A_{0},A_{0}A_{0}^{\prime},A_{0}^{\prime}A_{3},A_{0}^{\prime}A_{4}\}$ is a union of weighted Fermat-Steiner trees (Fermat-Steiner Frechet multitree) with corresponding weights\\ $\{B_{1},B_{2},\frac{B_{0}+B_{0}^{\prime}}{2},B_{3},B_{4}\}$ for the boundary tetrahedron $A_{1}A_{2}A_{3}A_{4},$ then we derive two boundary triangles with two isolated degrees of intelligence for all incogruent boundary tetrahedra (Frechet multitetrahedron).
\end{proposition}

\begin{proof}
The orientation \[A_{0}\to A_{1}\to A_{2} \to A_{0}\to A_{0}^{\prime}\to A_{3}\to A_{4}\to A_{0}^{\prime}\to A_{0}\]
in the sense of P. ALexandrov cancels the line segment $A_{0}$ and $A_{0}^{\prime}$ from the weighted Fermat-Steiner tree $\{A_{1}A_{0},A_{2}A_{0},A_{0}A_{0}^{\prime},A_{0}^{\prime}A_{3},A_{0}^{\prime}A_{4}\},$ which leads to two  boundaries $\triangle A_{1}A_{2}A_{0}$ and $A_{3}A_{4}A_{0}^{\prime},$ with two isolated degrees of intelligence at the vertices $A_{0}$ and $A_{0}^{\prime}.$ Therefore, taking into account the conditions for a sextuple of positive real numbers determining the edge lengths of a maximum of thirty incongruent tetrahedra (Frechet multitetrahedron) studied by Blumenthal, Herzog and Dekster Wilker, we may get a union of isolated boundary triangles, which correspond to the weighted Fermat-Steiner Frechet tree (multitree) for the Frechet multitetrahedron.
\end{proof}

In \cite{EdelsteinScwartz:76} and \cite{Eremenko:09}, Edelstein, Schwartz and Eremenko proved Gehrink's problem on linked curves in $\mathbb{R}^{3}.$

\begin{lemma}{\cite{EdelsteinScwartz:76}, \cite{Eremenko:09}}\label{gehringr3}
If $C_{1}$ and $C_{2}$ are two linked closed curves (any continuous deformation of $C_{1}$ to a point intersect $C_{2}$) in $\mathbb{R}^{3}$ and the distance between $C_{1}$ and $C_{2}$ is 1, then the length of $C_{1}$ or $C_{2}$ is at least $2\pi.$
\end{lemma}

\begin{proposition}[Gehring's linked curved inequality associated with the weighted Fermat-Steiner Frechet multitree in $\mathbb{R}^{3}$]\label{Gehringr3fermatsteiner}
If $A_{1}A_{2}A_{3}A_{4}$ belongs to a union of incongruent tetrahedra, which give the Frechet multitetrahedron,  formed by a given sextuple of positive real numbers, determining edge lengths, $C_{1}\equiv boundary(\triangle A_{1}A_{0}A_{0}^{\prime}),$ $C_{2}\equiv boundary (\triangle A_{2}A_{3}A_{4}),$ $A_{0},$ $A_{0}^{\prime}$ the corresponding Fermat-Steiner points and \\$dist(\triangle A_{1}A_{0}A_{0}^{\prime}, \triangle A_{2}A_{3}A_{4}))=1,$
then the length of $boundary(\triangle A_{1}A_{0}A_{0}^{\prime})$ or $boundary (\triangle A_{2}A_{3}A_{4})$
is at least $2\pi.$
\end{proposition}

\begin{proof}
By applying lemma!\ref{gehringr3}, for $C_{1}\equiv boundary(\triangle A_{1}A_{0}A_{0}^{\prime}),$ $C_{2}\equiv boundary (\triangle A_{2}A_{3}A_{4}),$ we derive that the length of $boundary(\triangle A_{1}A_{0}A_{0}^{\prime})$ or $boundary (\triangle A_{2}A_{3}A_{4})$ is at least $2\pi.$
\end{proof}

\begin{remark}
A generalization of Proposition~\ref{Gehringr3fermatsteiner} for Steiner trees in $\mathbb{R}^{3}$ may give a new perspective to the Steiner ratio conjecture, which was proved for manifolds by Ivanov-Tuzhlin-Cieslik in \cite{IvanovTuzhilinCieslik:03}.
\end{remark}

\section{Some computations on the variation of Fermat trees for tetrahedra having one or three vertices at infinity in $\mathbb{R}^{3}$}
In this section, we present a new class of weighted Fermat trees for a tetrahedron $A_{1}A_{2}A_{3}A_{4}$ having the vertex $A_{4}$ at infinity or the three vertices $A_{1},$ $A_{2},$ $A_{3}$  at infinity, such that $A_{4}$ belongs to the perpendicular line w.r to the plane defined by $\triangle A_{1}A_{2}A_{3}$ at the corresponding weighted Fermat point $A_{0,123}$ of $\triangle A_{1}A_{2}A_{3},$ which is derived by setting $B_{4}=0$ in the weighted Problem for $A_{1}A_{2}A_{3}A_{4}$ and the length of $A_{4}A_{0,123}$ is a large positive real number $M,$ for $B_{4}=1,$ and $B_{i}=B_{i}(M),$ are positive linear function w.r. to $M,$ for $i=1,2,3.$

Let $A_{1}A_{2}A_{3}A_{4}$ be a tetrahedron having the vertex $A_{4}$ at infinity in $\mathbb{R}^{3},$ with corresponding weights
$B_{1}(M)=b_{1}M+c_{1}>0,$ $B_{2}(M)=b_{2}M+c_{2}>0,$ and $B_{3}(M)=b_{3}M+c_{3}>0,$ $B_{4}=1$ having positive real values for a given large number $M$ and given real numbers $b_{i},$ $c_{i}$ for $i=1,2,3.$

We denote by $A_{0,123}$ the corresponding weighted Fermat-Torricelli point of $\triangle A_{1}A_{2}A_{3}$
for given weights $B_{1}(M),$ $B_{2}(M)$ and $B_{3}(M),$ which satisfy the inequalities (\ref{ineqq1}) of the weighted floating case of Theorem~\ref{theorrn1} by setting $B_{4}=0.$


We set $\varphi\equiv \angle A_{0,123}A_{1}A_{3},$ $a_{i,0123}\equiv \|A_{0,123}A_{i}\|$ and $B_{i}\equiv B_{i}(M),$ for $i=1,2,3.$

\begin{lemma}{\cite{Zachos:13}}\label{corlem1} The exact position of the weighted Fermat-Torricelli tree w.r to  $\triangle A_{1}A_{2}A_{3}$  is given by:
\begin{equation}\label{mainres}
\varphi=\operatorname{arccot}\left(\frac{\sin(\alpha_{213})-\cos(\alpha_{213})
\cot(\arccos{\frac{B_{3}^{2}-B_{1}^2-B_{2}^2}{2B_{1}B_{2}}})-
\frac{a_{13}}{a_{12}
}\cot(\arccos{\frac{B_{2}^{2}-B_{1}^2-B_{3}^2}{2B_{1}B_{3}}})}
{-\cos(\alpha_{213})-\sin(\alpha_{213})
\cot(\arccos{\frac{B_{3}^{2}-B_{1}^2-B_{2}^2}{2B_{1}B_{2}}})+
\frac{a_{13}}{a_{12} }}\right)
\end{equation}
and
\begin{equation}\label{a10123}
a_{1,0123}=\frac{\sin\left(\varphi+\arccos{\frac{B_{2}^{2}-B_{1}^2-B_{3}^2}{2B_{1}B_{3}}}\right)a_{13}}{\sin\left(\arccos{\frac{B_{2}^{2}-B_{1}^2-B_{3}^2}{2B_{1}B_{3}}}\right)}.
\end{equation}

\end{lemma}

From Lemma~\ref{corlem1}, we derive that:

\begin{lemma}\label{corol11}
The line segments $a_{2,0123}$ and $a_{3,0123}$ depend on $B_{1},$ $B_{2},$ $B_{3},$ $a_{12},$ $a_{13},$ $a_{23}$ and $\varphi:$
\begin{equation}\label{a20123}
a_{2,0123}=\sqrt{a_{1,0123}^2+a_{12}^2-2a_{1,0123}a_{12}\cos(\angle A_{2}A_{1}A_{3}-\varphi)}
\end{equation}
and
\begin{equation}\label{a30123}
a_{3,0123}=\sqrt{a_{1,0123}^2+a_{13}^2-2a_{1,0123}a_{13}\cos(\varphi)}.
\end{equation}

\end{lemma}

We assume that $A_{4}$ lies on the normal line w.r. to the plane defined by $\triangle A_{1}A_{2}A_{3},$
for $B_{4}=1$ and the the length of $A_{0,123}A_{4}$ is $M.$


We recall that the Cayley-Menger determinant $D(S)$ is given by:

\begin{equation}\label{CaleyMenger}
D(S) =\operatorname{det} \left(
\begin{array}{ccccc}
0      & a_{12}^2      & a_{13}^2 & a_{14}^2  &  1  \\
a_{12}^2 & 0 & a_{23}^2 &a_{24}^2 & 1          \\
a_{13}^2 & a_{23}^2 &0   &a_{34}^2  & 1         \\
a_{14}^2 & a_{24}^2 &a_{34}^2  & 0   & 1         \\
1 & 1 &1  & 1     &   0    \\
\end{array} \right).
\end{equation}

We recall that $a_{ij}$ is length of the line segment $A_iA_j,$
$\alpha_{ijk}\equiv \angle A_{i}A_{j}A_{k},$ the dihedral
angle $\alpha$ is defined by the planes formed by $\triangle
A_{0}A_{1}A_{2}$ and $\triangle A_{1}A_{2}A_{3},$  the dihedral
angle $\alpha_{g_{4}}$ is defined by the planes formed by $\triangle
A_{1}A_{2}A_{3}$ and $\triangle A_{1}A_{2}A_{4},$ $h_{0,12}$ is
the height of $\triangle A_{0}A_{1}A_{2}$ from $A_{0},$ by
$h_{0,12m}$ the distance of $A_{0}$ from the plane defined by
$\triangle A_{1}A_{2}A_{m},$ for $i,j,k=0,1,2,3,4$ and $m=3,4.$

The variable line segments $a_{03}(a_{01},a_{02},\alpha),$ $a_{04}(a_{01},a_{02},\alpha),$ are derived in \cite[Formulas (2.14), (2.20) p.~116]{Zach/Zou:09}.

\begin{proposition}\label{a34distanceM}

The variable lengths $a_{03}(a_{01},a_{02},\alpha),$ $a_{04}(a_{01},a_{02},\alpha;M),$ are given by:

\begin{equation}\label{impa03}
a_{03}(a_{01},a_{02},\alpha)=\sqrt{a_{02}^2 +a_{23}^2-2 a_{23}[\sqrt{a_{02}^2-h_{0,12}^2}
\cos\alpha_{123} +h_{0,12}\sin\alpha_{123}\cos\alpha ]}
\end{equation}

and

\begin{equation}\label{impa04}
a_{04}(a_{01},a_{02},\alpha;M)=\sqrt{a_{02}^2 +a_{24}^2-2 a_{24}[\sqrt{a_{02}^2-h_{0,12}^2}
\cos\alpha_{124}
+h_{0,12}\sin\alpha_{124}\cos(\alpha_{g_{4}}-\alpha) ]},
\end{equation}
where
\begin{equation}\label{height012}
h_{0,12}=h_{0,12}(a_{01},a_{02},a_{12})=\frac{a_{01}a_{02}}{a_{12}}\sqrt{1-\left(\frac{a_{01}^{2}+a_{02}^{2}-a_{12}^2}{2a_{01}a_{02}}
\right)^{2}}.
\end{equation}
such that:

\begin{equation}\label{lima4}
\lim_{M\to +\infty}a_{04}(a_{01},a_{02},\alpha;M)=+\infty.
\end{equation}

\end{proposition}

\begin{proof}

The angle $\alpha_{123}$
does not depend on $M:$
\begin{equation}\label{cosalapha123}
\cos\alpha_{123}=\frac{a_{12}^2+a_{23}^2-a_{13}^2}{2 a_{12}
a_{23}},
\end{equation}

and

\begin{equation}\label{sinalapha123}
\sin\alpha_{123}=\frac{\sqrt{(a_{12}+a_{23}+a_{13})(a_{23}+a_{13}-a_{12})(a_{12}+a_{13}-a_{23})(a_{12}+a_{23}-a_{13})}}{2
a_{12} a_{23}}
\end{equation}

By replacing (\ref{cosalapha123}), (\ref{sinalapha123}) and (\ref{height012}) in (\ref{impa03})
we derive that $a_{03}=a_{03}(a_{01},a_{02},\alpha).$


We shall show that $a_{41},$ $a_{42},$ $\alpha_{124}$ depend on $M.$

From the right triangles $\triangle A_{1}A_{0,123}A_{4},$ $\triangle A_{2}A_{0,123}A_{4}$ and taking into account
(\ref{a10123}) and (\ref{a20123}),  we get:

\begin{equation}\label{a41}
a_{41}=a_{41}(M)=\sqrt{M^{2}+a_{1,0123}^2}
\end{equation}

and

\begin{equation}\label{a42}
a_{42}=a_{42}(M)=\sqrt{M^{2}+a_{2,0123}^2}.
\end{equation}

where

\begin{equation}\label{a04newa01a02a03}
a_{4}^2=a_{2}^2 +a_{24}^2-2 a_{24}[\sqrt{a_{2}^2-h_{0,12}^2}
\cos\alpha_{124}
+h_{0,12}\sin\alpha_{124}(\cos\alpha_{g_{4}}\cos\alpha+\sin\alpha_{g_{4}}\sin\alpha
           ) ]
\end{equation}

where

\begin{equation}\label{cosalapha124}
\cos\alpha_{124}=\frac{a_{12}^2+a_{24}^2-a_{14}^2}{2 a_{12}
a_{24}},
\end{equation}

\begin{equation}\label{sinalapha124}
\sin\alpha_{124}=\frac{\sqrt{(a_{12}+a_{24}+a_{14})(a_{24}+a_{14}-a_{12})(a_{12}+a_{14}-a_{24})(a_{12}+a_{24}-a_{14})}}{2
a_{12} a_{24}},
\end{equation}


\begin{equation}\label{cosaphag4}
\alpha_{g_{4}}=\arccos\left(
\frac{\left(\frac{a_{42}^2+a_{23}^2-a_{43}^2}{2 a_{23}}
\right)-\sqrt{a_{42}^2-h_{4,12}^2}\cos\alpha_{123}}{h_{4,12}\sin\alpha_{123}}
\right)
\end{equation}

and

\begin{equation}\label{height412}
h_{4,12}=h_{4,12}(a_{41},a_{42},a_{12})=\frac{a_{41}a_{42}}{a_{12}}\sqrt{1-\left(\frac{a_{41}^{2}+a_{42}^{2}-a_{12}^2}{2a_{41}a_{42}}
\right)^{2}}.
\end{equation}

The dihedral angle $\alpha_{g_{4}}$ is derived by setting in
(\ref{impa03}) and (\ref{height012}) the index from $0\to 4.$


By replacing (\ref{a41}), (\ref{a42}), (\ref{cosalapha124}), (\ref{sinalapha124}), (\ref{cosaphag4}), (\ref{height012}) in (\ref{impa04})

 we derive that: $a_{04}=a_{04}(a_{01},a_{02},\alpha;M).$

\end{proof}


The position of the weighted Fermat-Torricelli
tree of a tetrahedron $A_{1}A_{2}A_{3}A_{4}$ is given by computing the
volumes of the tetrahedra $\operatorname{Vol}(A_{0}A_{i}A_{j}A_{k})$ for $i,j,k=1,2,3,4,$
via the Caley-Menger determinant
(\cite[pp.~249-255]{Uspensky:48}), which depend on $a_{1},$ $a_{2}$ ,$\alpha$ and $M.$

We set $C\equiv \frac{\sum_{i=1}^{4}\frac{B_{i}}{a_{i}}}{\operatorname{Vol}(A_{1}A_{2}A_{3}A_{4})}.$

\begin{theorem}\label{tetrahedroninf}
The following four equations provide a necessary condition to
determine the position of the weighted Fermat-Torricelli tree at the interior of $A_{1}A_{2}A_{3}A_{4}:$

\begin{equation}\label{tetraed1}
\left(\frac{B_{3}}{a_{3}\operatorname{Vol}(A_{0}A_{1}A_{2}A_{4})}\right)^{2}=\left(\frac{B_{4}}{a_{4}\operatorname{Vol}(A_{0}A_{1}A_{2}A_{3})}\right)^{2}=C^2,
\end{equation}

\begin{equation}\label{tetraed2}
\left(\frac{B_{3}}{a_{3}\operatorname{Vol}(A_{0}A_{1}A_{2}A_{4})}\right)^{2}=\left(\frac{B_{1}}{a_{1}\operatorname{Vol}(A_{0}A_{2}A_{3}A_{4})}\right)^{2}=C^2,
\end{equation}

\begin{equation}\label{tetraed3}
\left(\frac{B_{3}}{a_{3}\operatorname{Vol}(A_{0}A_{1}A_{2}A_{4})}\right)^{2}=\left(\frac{B_{2}}{a_{2}\operatorname{Vol}(A_{0}A_{1}A_{3}A_{4})}\right)^{2}=C^2,
\end{equation}

and

\begin{equation}\label{tetraed4}
\left(\frac{B_{1}}{a_{1}\operatorname{Vol}(A_{0}A_{2}A_{3}A_{4})}\right)^{2}=\left(\frac{B_{2}}{a_{2}\operatorname{Vol}(A_{0}A_{1}A_{3}A_{4})}\right)^{2}=C^2.
\end{equation}

where
\begin{equation}\label{BiM}
B_{i}(M)=b_{i}M+c_{i}>0,
\end{equation}
for $i=1,2,3$

and

\[B_{4}=1.\]

\end{theorem}

\begin{proof}

The objective function is given by:
\begin{equation}\label{mainobj1}
f(a_{1},a_{2},\alpha;M)=B_{1}a_{1}+B_{2}a_{2}+B_{3}a_{3}(a_{1},a_{2},\alpha)+B_{4}a_{4}(a_{1},a_{2},\alpha;M).
\end{equation}

or

\begin{equation}\label{mainobj2}
f(a_{1},a_{4},\alpha^{\prime};M)=B_{1}a_{1}+B_{4}a_{4}+B_{2}a_{2}(a_{1},a_{4},\alpha^{\prime})+B_{3}a_{3}(a_{1},a_{4},\alpha^{\prime};M).
\end{equation}

or

\begin{equation}\label{mainobj3}
f(a_{2},a_{3},\alpha^{\prime\prime};M)=B_{2}a_{2}+B_{3}a_{3}+B_{1}a_{1}(a_{2},a_{3},\alpha^{\prime\prime})+B_{4}a_{4}(a_{2},a_{3},\alpha^{\prime\prime};M).
\end{equation}

where $\alpha^{\prime}$ is the dihedral angle formed by $\triangle A_{1}A_{0}A_{4}$ and $\triangle A_{1}A_{4}A_{2}$
and $\alpha^{\prime\prime}$ is the dihedral angle formed by $\triangle A_{2}A_{0}A_{3}$ and $\triangle A_{2}A_{1}A_{3}.$

By differentiating (\ref{mainobj1}) w.r to $\alpha,$ (\ref{mainobj2}) w.r. to $\alpha^{\prime}$ and
(\ref{mainobj3}) w.r. to $\alpha^{\prime\prime}$
we derive (see also in \cite[Formula~(2.25),pp.~117]{Zach/Zou:09}):

\begin{eqnarray}\label{tetraed5}
&&\frac{B_{3}}{a_{3}\operatorname{Vol}(A_{0}A_{1}A_{2}A_{4})}=\frac{B_{4}}{a_{4}\operatorname{Vol}(A_{0}A_{1}A_{2}A_{3})}={}\nonumber\\
&&{}=\frac{B_{1}}{a_{1}\operatorname{Vol}(A_{0}A_{2}A_{3}A_{4})}=\frac{B_{2}}{a_{2}\operatorname{Vol}(A_{0}A_{1}A_{3}A_{4})}=C,
\end{eqnarray}

The volume $\operatorname{Vol}(A_{0}A_{i}A_{j}A_{k}),$ for $i,j,k=1,2,3,4$ is given by(\cite[pp.~249-255]{Uspensky:48}):

\begin{equation}\label{V0123}
288 \operatorname{Vol}(A_{0}A_{1}A_{2}A_{3})^{2}
=D(\{a_{1},a_{2},a_{3},a_{23},a_{13},a_{12}\})
\end{equation}

\begin{equation}\label{V0234}
288 \operatorname{Vol}(A_{0}A_{2}A_{3}A_{4})^{2}
=D(\{a_{4},a_{2},a_{3},a_{23},a_{43},a_{42}\})
\end{equation}

\begin{equation}\label{V0134}
288 \operatorname{Vol}(A_{0}A_{1}A_{3}A_{4})^{2}
=D(\{a_{1},a_{4},a_{3},a_{43},a_{13},a_{14}\})
\end{equation}

and

\begin{equation}\label{V0124}
288 \operatorname{Vol}(A_{0}A_{1}A_{2}A_{4})^{2}
=D(\{a_{1},a_{2},a_{4},a_{24},a_{14},a_{12}\}) .
\end{equation}

By squaring both parts of the equations in (\ref{tetraed5}) and
then by replacing (\ref{V0123}), (\ref{V0124}), (\ref{V0134}),
(\ref{V0234}) in the derived equations, we deduce
(\ref{tetraed1}), (\ref{tetraed2}), (\ref{tetraed3}) and
(\ref{tetraed4}) which depend on $a_{1},$ $a_{2},$ $\alpha$ and $M.$

\end{proof}

For $M\to +\infty,$ the solution of the weighted Fermat-Torricelli problem is a weighted Fermat-Torricelli tree
with branches $A_{1}A_{0,123},$ $A_{2}A_{0,123},$ $A_{3}A_{0,123}$ and $A_{4}A_{0,123}.$
We call the weighted Fermat-Torricelli tree for a tetrahedron having one vertex at infinity a large tree because one of the four branches $A_{4}A_{0,123}\to \infty.$

The unique solution of the inverse problem for tetrahedra in $\mathbb{R}^{3}$
has been established in \cite{Zach/Zou:09}.

\begin{problem}{Inverse weighted Fermat problem for tetrahedra in $\mathbb{R}^{3},$\cite{Zach/Zou:09}}\label{invFT}

Given a point $A_{0}$ and a positive real number $C$ which belongs to the interior of
$A_{1}A_{2}A_{3}A_{4}$ in $\mathbb{R}^{3}$, does there exist a
unique set of positive weights $B_{i},$ such that
\begin{displaymath}
 B_{1}+B_{2}+B_{3}+B_{4} = C,
\end{displaymath}
for which $A_{0}$ minimizes
\begin{displaymath}
 f(A_{0})=\sum_{i=1}^{4}B_{i}a_{0i}.
\end{displaymath}
\end{problem}

We denote by $\alpha_{i,j0k}$ the angle that is formulated by the line segment $A_{0}A_{i}$ and the line segment that connects $A_{0}$  with the trace of the orthogonal projection of $A_{i}$ to the plane defined by $\triangle A_{j}A_{0}A_{k}.$ A positive  answer w.r. to the inverse problem for $A_{1}A_{2}A_{3}A_{4}$ is
given in \cite[Proposition~1]{Zach/Zou:09}):

\begin{lemma}\cite[Proposition~1, Solution of Problem~2]{Zach/Zou:09}\label{leminv5}
The weight $B_{i}$ are uniquely determined by the formula:
\begin{equation}\label{inverse111}
B_{i}=\frac{C}{1+\|\frac{\sin{\alpha_{i,k0l}}}{\sin{\alpha_{j,k0l}}}\|+\|\frac{\sin{\alpha_{i,j0l}}}{\sin{\alpha_{k,j0l}}}\|+\|\frac{\sin{\alpha_{i,k0j}}}{\sin{\alpha_{l,k0j}}}\|},
\end{equation}
where

\begin{equation}\label{ratioji2}
\frac{B_{j}}{B_{i}}=\frac{\sin{\alpha_{i,k0l}}}{\sin{\alpha_{j,k0l}}}
\end{equation}

\begin{equation}\label{ratioji}
\frac{\sin{\alpha_{i,k0l}}}{\sin{\alpha_{j,k0l}}}=\sqrt{ \| \frac{\sin^{2}\alpha_{k0m}-\cos^{2}\alpha_{m0i}-\cos^{2}\alpha_{k0i}+2\cos\alpha_{k0m}\cos\alpha_{m0i}\cos\alpha_{k0i}}{
\sin^{2}\alpha_{k0m}-\cos^{2}\alpha_{m0j}-\cos^{2}\alpha_{k0j}+2\cos\alpha_{k0m}\cos\alpha_{m0j}\cos\alpha_{k0j}} \| }
\end{equation}
for $i,j,k,l=1,2,3,4$ and $i \neq j\neq k\neq l.$

\end{lemma}

For $B_{4}=0,$ we obtain the inverse weighted Fermat-Torricelli problem for $\triangle A_{1}A_{2}A_{3}.$

\begin{lemma}\cite{Gue/Tes:02}\label{leminv6}
The weight $B_{i}$ are uniquely determined by the formula:
\begin{equation}\label{inverse111}
B_{i}=\frac{C}{1+\frac{\sin{\alpha_{i0k}}}{\sin{\alpha_{j0k}}}+\frac{\sin{\alpha_{i0j}}}{\sin{\alpha_{j0k}}}},
\end{equation}
where

\begin{equation}\label{ratioji2}
\frac{B_{j}}{B_{i}}=\frac{\sin{\alpha_{i0k}}}{\sin{\alpha_{j0k}}}.
\end{equation}

for $i,j,k=1,2,3.$
\end{lemma}

\begin{theorem}\label{theor2inf}
If $M\to +\infty,$ for $B_{4}=1$ the solution of the inverse problem for
the tetrahedron $A_{1}A_{2}A_{3}A_{4}$ having the vertex $A_{4}$ at infinity in $\mathbb{R}^{3}$ coincides with the solution of the inverse problem for $\triangle A_{1}A_{2}A_{3}.$
\end{theorem}

\begin{proof}
For $M\to +\infty,$ $A_{0}\to A_{0,123}$ and $\angle A_{4}A_{0}A_{i}=90^{\circ},$ for $i=1,2,3.$
By replacing $\angle A_{4}A_{0}A_{1}=\angle A_{4}A_{0}A_{2}=\angle A_{4}A_{0}A_{3}=90^{\circ}$
in (\ref{ratioji}) and (\ref{ratioji2}) for $i,j=1,2,3,$ $k,l,m=1,2,3,4$ and $i \neq j\neq k\neq l,$
we derive $B_{i}$
\[\frac{B_{j}}{B_{i}}=\frac{\sin{\alpha_{i0k}}}{\sin{\alpha_{j0k}}},\]

for $i,j,k=1,2,3$ and $B_{4}=1.$

\end{proof}

A direct consequence of Theorem~4, for

\[B_{1}(M)+B_{2}(M)+B_{3}(M)+B_{4}(M)=C\]

($C$ is a constant real number independent of $M$)

gives the following Proposition (the weights $B_{i}$ are independent of $M$)
\begin{proposition}\label{sumbconstant}
For $M\to +\infty$ and

\[B_{i}(M)=b_{i}M+c_{i}>0,\]

\[B_{4}=1,\]

for $i=1,2,3$
such that

\[b_{1}+b_{2}+b_{3}=0,\]
$B_{i}$ are uniquely determined by the formula:
\begin{equation}\label{inverse111M}
B_{i}=\frac{C}{1+\frac{\sin{\alpha_{i0k}}}{\sin{\alpha_{j0k}}}+\frac{\sin{\alpha_{i0j}}}{\sin{\alpha_{j0k}}}},
\end{equation}
for $C=1+c_{1}+c_{2}+c_{3}$
and $i,j,k=1,2,3.$
\end{proposition}

As a future work, we may focus on studying properties of the Betti group of a Frechet multisimplex.
\\
\\
\\


\end{document}